\documentclass[12pt, a4paper]{amsart}

\usepackage{amsmath}
\usepackage{amssymb}
\usepackage{amsthm}
\usepackage{amscd}
\usepackage{amsfonts}
\usepackage{amsbsy}
\usepackage{epsfig}
\usepackage{enumerate}
\usepackage{soul}
\usepackage[normalem]{ulem}
\usepackage[shortlabels]{enumitem}
\usepackage[latin1]{inputenc}  
\usepackage{tikz}
\newtheorem{theorem}{Theorem}[section]
\newtheorem{lemma}[theorem]{Lemma}

\newtheorem{proposition}[theorem]{Proposition}
\newtheorem{corollary}[theorem]{Corollary}

\newtheorem{remark}[theorem]{Remark}

\newtheorem{definition}[theorem]{Definition}

\newtheorem*{claim*}{Claim}

\newfont\bbf{msbm10 at 12pt}

\def\eps{\varepsilon}

\def\N{{\mathbb N}}

\def\J{{\mathcal J}}

\def\R{{\mathcal R}}

\def\blue{\color{blue}}
\def\red{\color{red}}

\def\supp{\mbox{\rm supp}}

\def\diam{\mbox{\rm diam} }
\def\sgn{\mbox{\rm sgn} }

\def\le{\leqslant}
\def\ge{\geqslant}

\def\per{\mathrm{Per}}

\def\1{ {\hbox{{\it 1}} \!\! I} }

\def\Int{\mathrm{Int}}
\def\CP{\overline{CP}}
\newcommand{\fix}{\mathrm{Fix}}
\def\gh{\hat{g}}

\begin{document}

\title[Interval maps with dense periodicity]
{Interval maps with dense periodicity}
\author{Jozef Bobok}
\author{Jernej \v Cin\v c}
\author{Piotr Oprocha}
\author{Serge Troubetzkoy}

\date{\today}

	\begin{abstract}
We consider the class of interval maps with dense set of periodic points $CP$ and its closure $\overline{CP}$ equipped with the metric of uniform convergence. Besides studying basic topological properties and density results in the spaces $CP$ and $\overline{CP}$ we prove that $\overline{CP}$ is dynamically characterized as the set of interval maps for which every point is chain-recurrent. Furthermore, we prove that a strong topological expansion property called topological exactness (or leo property) is attained on the open dense set of maps in $CP$ and on a residual set in $\CP$. Moreover, we show that every second category set in $CP$ and $\CP$ is rich in a sense that it contains uncountably many conjugacy classes. 
 An analogous conclusion also holds in the setting of interval maps preserving any fixed non-atomic probability measure with full support.
Finally, we give a detailed description of the structure of periodic points of generic maps in $CP$ and $\CP$ and show that generic maps in  $CP$ and $\CP$ satisfy the shadowing property.
	\end{abstract}

\address[J.\ Bobok]{Department of Mathematics of FCE, Czech Technical University in Prague,
Th\'akurova 7, 166 29 Prague 6, Czech Republic}
\email{jozef.bobok@cvut.cz}

\address[J.\ \v{C}in\v{c}]{University of Maribor, Koro\v ska 160, 2000 Maribor, Slovenia
	-- and --
	Abdus Salam International Centre for Theoretical Physics (ICTP), Trieste, Italy}
\email{jernej.cinc@um.si}

\address[P.\ Oprocha]{AGH University of Krakow, Faculty of Applied Mathematics,
al.\ Mickiewicza 30, 30-059 Krak\'ow, Poland. -- and --
Centre of Excellence IT4Innovations - Institute for Research and Applications of Fuzzy Modeling, University of Ostrava, 30. dubna 22, 701 03 Ostrava 1, Czech Republic}
\email{oprocha@agh.edu.pl}

\address[S.\ Troubetzkoy]{Aix Marseille Univ, CNRS, I2M, Marseille, France}
\email{serge.troubetzkoy@univ-amu.fr}


\maketitle
\section{Introduction}
Many definitions of chaos in a compact metric space $X$ involve an assumption that the set of periodic points is dense in $X$.  In the current article we put ourselves in the context of $X$ being a unit interval and study dynamical, topological and dimension properties  of the set of interval maps with dense set of periodic points $CP$ and its closure $\CP$ using the metric of uniform convergence. Our line of study can be thought as a continuation of Barge and Martin's study \cite{BM}, where the authors gave a nice dynamical characterization of the maps from $CP$. 

Let $I:=[0,1]$ and let $C(I)$ denote the set of all continuous interval maps. We call $f\in C(I)$ a chain recurrent map if $f$ is chain recurrent in every point of $I$. The notion of a chain recurrent map is one of the standard notions in topological dynamics. In \cite{BC} Block and Coven provided a dynamical description of chain recurrent maps. In complement to these results, 
our first main result shows that a map $f \in C(I)$ is chain-recurrent if and only if $f\in \overline{CP}$ (Theorem~\ref{ChainRecurrentThenCPClosure}). 
Theorem~\ref{ChainRecurrentThenCPClosure} implies that the maps $\mathrm{id}$ and up to conjugacy the map $1-\mathrm{id}$  are the only maps with zero topological entropy from $\overline{CP}$ (Corollary~\ref{cor:turbulence}). A map $f\in C(I)$ is called \textit{turbulent} if there exist nondegenerate closed intervals $J, K \subset I$ with at most one point in common, such that $J\cup K\subset f(J)\cap f(K)$.

 Note that Barge and Martin in \cite{BM} proved that if $f$ has a dense set of periodic points, then either $f^2$ is the identity map or $f^2$ is turbulent.
Theorem~\ref{ChainRecurrentThenCPClosure} together with results from \cite{BC} also implies that if $f\in \overline{CP}$ such that $h_{\mathrm{top}}(f)>0$ then $f^2$ is turbulent. Therefore
$h_{top}(f)\geq\log{2}/2$ for every map
$f \in \CP$ except for the $id$ or maps conjugate to $1-id$.

Our second main result is that 
$CP$ and $\overline{CP}$ are uniformly locally arcwise connected (see the beginning of Subsection~\ref{subsec:locallyconnected} and Theorem \ref{thm1}). In particular, uniform local arcwise connectedness implies local connectedness. 
In the same section we also prove that the set of Lebesgue measure-preserving interval maps $C_{\lambda}$ is  arcwise connected (Corollary~\ref{cor:ac}), 
which, in particular, answers on a question of Micha\l\/  Misiurewicz asked during a conference in Barcelona in 2020. These results fit in the context of study initiated by Kolyada, Misiurewicz and Snoha \cite{KMS} who proved that the spaces of transitive, piecewise linear transitive and piecewise monotone transitive interval maps are locally connected and arcwise connected.  Some of these results were recently strengthened by He, Li and Yang in \cite{HLY}.

Next we turn our attention towards studying dynamical properties which hold on dense sets in the sets $CP$ and $\CP$. We prove another fundamental result that locally eventually onto (or leo for short, sometimes also called topologically exact) maps form an open dense collection of maps in $CP$ (Theorem~\ref{thm:leo}). This, in particular, gives another aspect to the dynamical characterization of maps from $CP$ from the paper of Barge and Martin \cite{BM}. Theorem~\ref{thm:leo} implies that there is a residual (but not open) set of leo maps in $\CP$ as well. To put these results in the context, it was proven by the first and the last author in \cite{BT} that generic interval maps of $C_{\lambda}$ are also leo. The proof of Theorem~\ref{thm:leo} also yields that leo maps are open and dense in $C_{\lambda}$.  As an application we get that maps satisfying the periodic specification property form an open dense collection in $CP$ (Corollary~\ref{cor:psp}).

 The study of the group of homeomorphisms and conjugacy class dates back to Rokhlin.
 Let $\mathcal{X}$ be a standard Borel space and $\nu$ a non-atomic probablity measure on $\mathcal{X}$.
Let $\mathrm{Aut}(\nu)$ denote the topological group of all measurable, measure-preserving bijections of $\mathcal{X}$.
A classical theorem of Rokhlin  says that the conjugacy classes on $\mathrm{Aut}(\nu)$ are meager (see \cite[Theorem 2.5]{Kech}). Moreover, Foreman and Weiss \cite{FW} proved that equivalence relation of conjugacy among the generic element in $\mathrm{Aut}(\nu)$ is non-classifiable. This line of research culminated recently in the paper of Solecki \cite{Sol} who answered the question of Glasner and Weiss about the structure of the closure of subgroups of $\mathrm{Aut}(\nu)$ generated by iterates of elements $T\in \mathrm{Aut}(\nu)$.

 We study analogous questions in the setting of interval map where the conjugating maps are homeomorphisms. For $f \in C(I)$ we consider its {\it conjugacy class}  defined by 
$$\{\psi^{-1} \circ f \circ \psi \in C(I) : \psi \text{ is a homeomorphism of } I\}.$$ 
Before studying other typical properties of the set $\CP$ we ask ourselves a fundamental question motivated by the preceding paragraph about the number of conjugacy classes in the residual sets in $\CP$. A beautiful result of Kechris and Rosendal \cite{KR} shows that there is a residual set in the space of all homeomorphisms of the Cantor set such that each pair of its elements are conjugate (so-called \textit{residual conjugacy class}). Moreover, Bernardes and Darji \cite{BD} proved there is a residual conjugacy class in the space of all self-maps of the Cantor set. In the case of $\CP$ all the standard invariants (entropy (see Proposition~\ref{prop:entropyinfty}), mixing,...) do not exclude the possibility that such a class exists.

Our fourth main result, Theorem~\ref{conj-nowhere}, is that every conjugacy class in $\CP$ is nowhere dense.
This implies that a residual conjugacy class cannot exist in $\CP$, it also implies that 
every second category set $G\subset \CP$ contains uncountably many conjugacy classes (Corollary~\ref{cor:conjugacyclasses}).   An analogous conclusion also holds in the setting of interval maps preserving any fixed non-atomic probability measure with full support and in $CP$.
Corollary~\ref{cor:conjugacyclasses} also 
has an application to the results in \cite{CO}; it
implies that there are uncountably many dynamically non-equivalent attractors in the neighborhood of any attractor 
in the construction of a
parameterized family of planar homeomorphisms with pseudo-arc attractors in \cite{CO}.

Finally we study typical properties of maps from  $CP$ as well as maps from $\CP$. A result in this direction was given in the literature recently in \cite{CO}, where the second and third author proved that a very strong condition on the values of interval maps called the $\delta$-crookedness is satisfied for some iterate of any generic map from $\CP$. This result can be interpreted through the notion of inverse limits; namely, the inverse limit of a generic map in $\CP$ is the pseudo-arc (see \cite[Theorem 1.3]{CO}). Similarly as in the context of $C_{\lambda}$ in \cite{BCOT} we obtain that the shadowing property is typical in $CP$ and $\CP$. Furthermore, we prove that these results are optimal in $CP$ and $\CP$ since there is a dense set of maps in $CP$ resp. $\CP$ which do not have the shadowing property (see Theorem~\ref{thm:shadowing}). We also provide a study of the structure and dimension of periodic points of generic maps in  $CP$ and $\CP$ (related results in the setting of interval continuous map have been given by Agronsky, Bruckner and Laczkovicz \cite{ABL}). Namely, in Theorem~\ref{thm:periodicpoints} we completely describe the structure of fixed points, periodic points and the union of all periodic points as well as their Hausdorff, upper box and lower box dimension. Finally, we also show in Theorem~\ref{t-PP} that the set of leo maps in $CP$ whose periodic points have full Lebesgue measure and whose periodic points of period $k$ have positive measure for each $k \ge 1$ is dense in $CP$ to give a counterpart of Theorem~\ref{thm:periodicpoints}. 

Let us give a brief outline of the paper. In Section~\ref{subsec:preliminaries} we start with introducing the basic notions that are used throughout the article.
In Section~\ref{sec:characterizationCP} we prove Theorem~\ref{ChainRecurrentThenCPClosure} which gives a dynamical characterization of $\CP$.
We start Section~\ref{sec:structure} with Subsection~\ref{subsec:basicproperties} where we prove structure statements that $CP$ is a residual  subset of $\CP$ and that $\CP\setminus CP$ is dense in $\CP$. We continue with Subsection~\ref{subsec:locallyconnected} where we prove the results about topological structure of spaces $CP$, $\CP$ and $C_{\lambda}$, namely Theorem~\ref{thm:ac}, Corollary~\ref{cor:ac} and Theorem~\ref{thm1}. In Section~\ref{sec:densness} we address denseness properties of $CP$ and $\CP$. First, in Subsection~\ref{subsec:leo} we prove in Theorem~\ref{thm:leo} that the set of leo maps is open and dense in $CP$ and then in Subsection~\ref{subsec:conjugacyclasses} we show through Corollary~\ref{cor:conjugacyclasses} the result about conjugacy classes in $\CP$. In Subsection~\ref{subsec:windowperturbations} we introduce the technique called window perturbations which is used in the rest of the article. We finish the article with Subsections~\ref{subsec:shadowing} and \ref{subsec:periodicpoints} where we prove that shadowing is a typical property in $\CP$ (Theorem~\ref{thm:shadowing}) and describe the structure and dimension of periodic points of generic maps in $\CP$ (Theorem~\ref{thm:periodicpoints}).

\section{Preliminaries}\label{subsec:preliminaries}

Let $I := [0,1]$ denote the unit interval, $d$ denote the 
Euclidean distance on $I$, and  $C(I)$ denote the set of all surjective continuous interval maps.  A point $x$ is called \emph{periodic of period $N\in \mathbb{N}$}, if $f^N(x)=x$ and $f^i(x)\neq x$ for $1\leq i< N$. Periodic points of $f\in C(I)$ of period $N$ are denoted by $\per(f,N)$ and the set of periodic points of $f\in C(I)$ by $\per(f)$. Let $\lambda$ denote the Lebesgue measure on the unit interval $I:=[0,1]$. We denote by $C_{\lambda}\subset C(I)$ the family of all continuous Lebesgue measure-preserving functions of $I$  and more generally by $C_{\mu}\subset C(I)$ the family of all continuous interval maps preserving a non-atomic probability measure $\mu$ on $I$ with full support. For $\mu$ a non-atomic probability measure on $I$ with full support the map $\psi\colon~I\to I$ defined as 

$$\psi(x)=\mu([0,x])$$ 

is a homeomorphism of $I$; moreover, if $f$ preserves $\mu$ then $\psi\circ f\circ \psi^{-1}\in C_{\lambda}$ (see Remark in \cite{BCOT}). 

\begin{remark}\label{rem:psihomeo}
 The map $f\mapsto \psi\circ f\circ \psi^{-1}$ is a homeomorphism of $C_{\mu}$ with $C_{\lambda}$.
\end{remark}

Let $CP \subset C(I)$ denote the set of continuous maps with a dense set of periodic points, i.e., all maps $f\in C(I)$ such that the closure $\overline{\per(f)}=I$.
We consider the uniform metric on $C(I)$ defined by 
$$\rho (f,g) := \sup_{x \in [0,1]} |f(x) - g(x)|.$$

Let $B(x,\xi)$  denote the open ball of radius $\xi$ centered at the point $x$ in a metric space $X$ and for a set $U\subset X$ we shall denote 
$$B(U,\xi):=\bigcup_{x\in U}B(x,\xi).$$

We call a set containing a dense $G_{\delta}$ set {\em residual} and call a property {\em generic} if it is attained on at least a residual set of the Baire space on which we work.

 \begin{definition}
 We say a map $f\in C(I)$ is
\begin{itemize}
	\item {\em transitive} if for all nonempty open sets $U,V\subset I$ there exists $n\ge 0$ so that $f^n(U)\cap V\neq\emptyset$,
	\item {\em topologically mixing} if for all nonempty open sets $U,V\subset I$ there exists $n_0\geq0$ so that $f^n(U)\cap V\neq\emptyset$ for every $n\ge
	n_0$,
	\item   {\em locally eventually onto (leo)} (also known as {\em topologically exact}) if for every nonempty open set $U\subset I$ there exists $n\in{\mathbb N}$ so that $f^n(U)=I$.
\end{itemize}
\end{definition}

\section{A dynamical characterization of $\overline{CP}$}\label{sec:characterizationCP}
 For $f \in C(I)$, $x, y \in I$ and $\eps>0$ an \textit{$\eps$-chain} from $x$ to $y$ is a finite sequence
$x = x_0,\ldots, x_n = y$ where $n>0$ and $|f(x_i)-x_{i+1}| < \eps$ for $0<i< n-1$. We say
that $x$ is \textit{chain-recurrent (for $f$)} if for every $\eps>0$ there is an $\eps$-chain from $x$ to itself. We will call $f$ {\it chain-recurrent} if every point in $I$ is chain-recurrent for $f$. A point $x\in I$ is called {\it non-wandering} if for any open set $U$ containing $x$  there exists $n>0$ such that $f^n(U)\cap U\neq \emptyset$.  It follows from the definitions that every non-wandering point is chain-recurrent and that any chain-recurrent map is onto.  

\begin{theorem}\label{ChainRecurrentThenCPClosure}
A map $f \in C(I)$ is chain-recurrent if and only if $f\in \overline{CP}$.
\end{theorem}

\begin{proof}
$(\Rightarrow)$ Fix such a map $f$. The map $f$ has to be onto. Fix $\eps>0$. 
 By our assumption on $f$, there exist points $0=x^1_0,\dots,x^m_0=1\in I$ and  finitely many disjoint $\eps/6$-chains $\mathcal A_{\ell}:=\{x^{\ell}_0,x^{\ell}_1,x^{\ell}_2,\ldots,x^{\ell}_{n(\ell)-1},x^{\ell}_{n(\ell)}=x^{\ell}_0\}$ from $x^{\ell}_0$ to itself.  Since $\mathcal A_{\ell}$ are disjoint, the set $\mathcal A=\bigcup_{1\le \ell\le m}\mathcal A_{\ell}$ can be written spatially as $\mathcal I=\{0=i_0<i_1<\cdots<i_{n-1}=1\}$.
We can assume that 
\begin{itemize}
\item[(i)] $\vert f(u)-f(v)\vert<\eps/6$ for each $k\in \{0,\dots,n-2\}$ and points $i_k\le u<v\le i_{k+1}$.
 \item[(ii)] $\vert i_k-i_{k+1}\vert<2\eps/3$ for each $k\in \{0,\dots,n-2\}$.
\end{itemize}
 Let $\varphi\colon~\mathcal A\to \mathcal A$ satisfy $\varphi(x^{\ell}_i)=x^{\ell}_{i+1 \pmod{n(\ell)}}$. Consider the extension $g\colon~I\to I$ defined as the connect-the-dots map of $(\mathcal A,\varphi)$. Clearly, $\mathcal A$ is a union of $n(\ell)$-cycles for $g$, $1\le \ell\le m$. From (i) and the fact that $\mathcal A$ is a union of $\eps/6$-chains for $f$ we obtain for each $x \in (i_k,i_{k+1})$
\begin{align}&\label{a:1}\vert g(i_k)-g(x)\vert\le\vert g(i_k)-g(i_{k+1})\vert\le \vert g(i_k)-f(i_k)\vert+\vert f(i_k)-f(i_{k+1})\vert+ \\
&+\vert
f(i_{k+1})-g(i_{k+1})\vert<\eps/6+\eps/6+\eps/6<\eps/2\nonumber,
\end{align}
hence
\begin{align*}&\vert f(x)-g(x)\vert\le \vert f(x)-f(i_k)\vert+\vert f(i_k)-g(i_k)\vert+ \vert
g(i_k)-g(x)\vert<\\&<\eps/6+\eps/6+\eps/2=5\eps/6;
\end{align*}
it means that 
\begin{equation}\label{e:dist}\rho(f,g) <5\eps/6.
    \end{equation}

Denote $I_k :=[i_k,i_{k+1}]$, $0\le k\le n-2$. We can consider a regular piecewise affine perturbation $h\colon~I\to I$ of $g$ satisfying (see Figure \ref{fig:perturbation}).

\begin{itemize}
\item[(iii)] $g\vert_{\mathcal A}=h\vert_{\mathcal A}$,
\item[(iv)] for every non-empty interval $J$, either $h(J)\supset I_k$ for some $k$ or $\lambda(h(J))>\alpha\lambda(J)$ for some $\alpha>1$, 
\item[(v)] for each $k$, $h(I_k)$ contains $\mathcal I$-neighbors of $g(I_k)$,
\item[(vi)] $\vert g(x)-h(x)\vert<7\eps/6,~x\in I.$
\end{itemize}
The property (vi) can be fulfilled since (ii) and (\ref{a:1}). We get from (\ref{e:dist}) and (vi) 
\begin{equation}
\vert f(x)-h(x)\vert\le \vert f(x)-g(x)\vert+\vert g(x)-h(x)\vert<5\eps/6+7\eps/6=2\eps,~x\in I.
    \end{equation}
    
    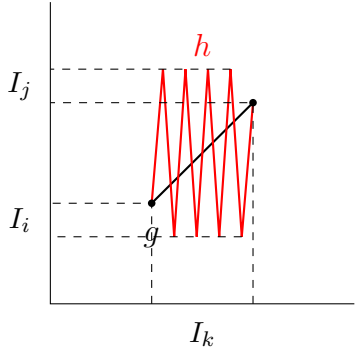
\begin{figure}
       \begin{tikzpicture}[scale=4]
           \draw (0,1)--(0,0)--(1,0);
           \draw[thick] (1/3,1/3)--(2/3,2/3);
           \draw[thick, red] (1/3,1/3)--(10/27,7/9)--(11/27,2/9)--(12/27,7/9)--(13/27,2/9)--(14/27,7/9)--(15/27,2/9)--(16/27,7/9)--(17/27,2/9)--(2/3,2/3);           	\node[circle,fill, inner sep=1] at (1/3,1/3){};
           	\node[circle,fill, inner sep=1] at (2/3,2/3){};
	        \node at (1/3,2/9) {$g$};
	        \node at (1/2,6/7) {\red $h$};
	        \draw[dashed] (1/3,0)--(1/3,1/3)--(0,1/3);
	        \draw[dashed] (2/3,0)--(2/3,2/3)--(0,2/3);
	        \draw[dashed] (17/27,2/9)--(0,2/9);
	        \draw[dashed] (16/27,7/9)--(0,7/9);
	        \node at (-0.1,5/18) {\small $I_i$};
	        \node at (-0.1,13/18) {\small $I_j$};
	        \node at (1/2,-0.1) {\small $I_k$};
       \end{tikzpicture}
       \caption{The perturbation $h$ of $g$.}\label{fig:perturbation}
    \end{figure}
    
    Let us prove that $h$ is leo. Fix a non-empty interval $J$, Property (iv) implies that after a finite number $m$ of iterations $h^m(J) \supset I_k$ for some $k$.  
 Remembering that $n=\Pi_{\ell}n(\ell)$ is the common period of our cycles, properties (iii) and (v) imply that applying $h^n$ to this $I_k$ must contain $I_k$ and its $\mathcal{I}$-neighbors. 
     Since there are finitely many intervals we eventually  cover the whole space.
    The claim is proven,  $h$ is a leo map, in particular $h$ is a map from the set $CP$. Since $\eps$ is arbitrary we have $f \in \overline{CP}$.  
    
   $(\Leftarrow)$ Fix a map $f\in \overline{CP}$ and $\eps>0$. Since $f$ is uniformly continuous, there is  $\delta>0$ such that $\vert f(u)-f(v)\vert<\eps/2$ whenever $\vert u-v\vert<\delta$. Fix an arbitrary point $x\in I$. Taking a map $g\in CP$ such that $\rho(f,g)<\eps/2$ and a $g$-periodic point $p$ of period $n$ satisfying $\vert x-p\vert<\delta$, we put  
$$x_0=x,~x_i=g^i(p),~ 1\le i\le n-1,~x_n=x.$$ Then 
$$\vert f(x_0)-x_1\vert\le \vert f(x)-f(p)\vert +\vert f(p)-g(p)\vert<\eps/2+\eps/2=\eps,$$
for each $1\le i\le n-2$, 
$$\vert f(x_i)-x_{i+1}\vert=\vert f(g^i(p))-g(g^i(p))\vert<\eps/2<\eps;$$
choosing $\delta<\eps/2$, since $p=g^n(p)$ and $x_n=x$, also 
$$\vert f(x_{n-1})-x_n\vert\le \vert f(x_{n-1})-p\vert +\vert p-x_n\vert<\eps/2+\eps/2=\eps.$$
Thus, the finite sequence $x_0=x,x_1=g(p),\dots,x_{n-1}=g^{n-1}(p),x_n$ is an $\eps$-chain from $x$ to itself.
\end{proof}

A map $f\in C(I)$ is called \textit{turbulent} if there exist nondegenerate closed intervals $J, K \subset I$ with at most one point in common, such that $J\cup K\subset f(J)\cap f(K)$. 

\begin{corollary}\label{cor:turbulence} The following hold:
\begin{enumerate}
\item The maps $\mathrm{id}$ and up to conjugacy the map $1-\mathrm{id}$  are the only maps with zero topological entropy from $\overline{CP}$.
\item
$f\in \overline{CP} \text{ such that } h_{\mathrm{top}}(f)>0 \implies f^2 \text{ turbulent.}$ Therefore, either $h_{top}(f)=0$ or $h_{top}(f)\geq\log{2}/2$. In the latter case $f$ has a periodic point of period $6$.
\end{enumerate}
\end{corollary}

\begin{proof}
 The statements follow from \cite{BC}, where the authors prove that interval map $f$ with chain-recurrent set $I$ is either such that $f^2=\mathrm{id}$  or $f^2$ is turbulent. Furthermore,  turbulent maps have  topological entropy at least   $\log 2$. But if $f^2=\mathrm{id}$ then either $f=id$ or   $f$ is topologically conjugate to $1-\mathrm{id}$. By the definition of turbulence, the second iterate of $f$ has a horseshoe and thus also the final of the claims above follows.
\end{proof}

 \begin{theorem}\label{thm:semi-conj}
If a map $f\in C(I)$ satisfies $\psi \circ f= g\circ \psi  \in C(I)$ where  map $g\in CP$ and where $\psi$ is (not necessarily strictly) monotone, but the $\psi$-preimages of $g$-periodic points  are singletons, then $f$ is chain recurrent. In particular, $f\in \CP$.
\end{theorem}
\begin{proof}
Fix any $x\in I$. If $x$ is in a singleton fibre of $\psi$ then it is approximated by periodic points since $g\in CP$ and hence belongs to a chain recurrent set. So assume that $x\in J=\psi^{-1}(z)$ for some $z$ where $\diam (J)>0$. Since $z$ is not a periodic point, for any $\eps>0$ there is $k$ and sets $J_i\subset I$, $i=-k,\ldots,k$ such that $\diam(J_{-k})<\eps$, $\diam(J_{k})<\eps$ and $f(J_i)=J_{i+1}$ for $i=-k,\ldots, k-1$ and $x\in J_0\subset J$. But then there is also a sequence of periodic points $p_n$ such that $\lim_{n}p_n\in J_{-k}$. Then it is not hard to see that there is a $2\eps$ chain from $x$ to $x$. Indeed, we iterate $x$ until reaching $J_{k}$, then we move onto orbit of $f^{2k}(p_n)$ and then return from $p_n$ to $J_{-k}$ hitting a point in $f^{-k}(x)\cap J_{-k}\neq \emptyset$. This can be done for every $\eps>0$ showing that $x$ is chain-recurrent.
\end{proof}

\section{Topological structure of the sets $CP$ and $\CP$}\label{sec:structure}

\subsection{Basic properties of $CP$, $\CP$ and $\CP\setminus CP$}\label{subsec:basicproperties}

 By definition $\CP$ is a complete metric space, thus a Baire space.
 The next proposition implies that $CP$ equipped with uniform metric is also a Baire space, see \cite[p. 10]{HM}.

\begin{lemma}\label{l1}
$CP$ is a residual subset of $\CP$.
\end{lemma}
\begin{proof}
Fix a sequence $(J_m)_m$ of rational subintervals of $[0,1]$
and a sequence $(f_n)_n$ of maps from $CP$ dense in $\CP$. 
By \cite[Remark p.~2]{BCOT} each map $f_n$ from $CP$ has an invariant measure $\mu_n$ with the full support $[0,1]$. 
In this proof we will use the notation $B_{\CP}(\cdot,\cdot)$ to denote an open ball in $\CP$.
Using the increasing homeomorphism $\psi_n$ of $[0,1]$ defined by $\psi_n(x)=\mu_n([0,x])$, we get that $g_n=\psi_n\circ f_n\circ \psi_n^{-1}\in C_{\lambda}$ and
$$
 U_n:=\{\psi_n\circ g\circ \psi_n^{-1}: g\in B_{\CP}(f_n,1/n)\}
$$
 is an open neighborhood of $g_n$ in $\CP$.
Since $C_\lambda \subset CP$ the set $U_n \cap CP$ contains a small  $C_\lambda$ neighborhood  of $g_n $.
 Repeating the construction from \cite[Lemma 12]{BCOT}, we can consider a piecewise affine map $g\in U_n\cap C_{\lambda}$ having a transversal periodic point in $\psi_n(J_m)$. Then for a sufficiently small $\delta$ positive, all maps in the ball $B_{\CP}(g,\delta)\subset U_n$ have a periodic point in $\psi_n(J_m)$. Put
$$
V_{m,n}:=\{\psi_n^{-1}\circ g'\circ \psi_n: g' \in B_{\CP}(g,\delta)\}.
$$
Clearly $V_{m,n}\subset B_{\CP}(f_n,1/n)$ is open and all maps in $V_{m,n}$ have a periodic point in $J_m$.  Moreover, the set $V_m :=\bigcup_n V_{m,n}$ is open and dense in $\CP$ containing maps with periodic points in $J_m$. It follows that any map from $\bigcap_mV_m$ contains periodic points in all $J_m$, and thus is in  $CP$.
At the same time $\bigcap_mV_m$ is a dense $G_{\delta}$ in $\CP$.  
\end{proof}

\begin{corollary}\label{cor:resinCPresinCPclosure}
If $A$ is a residual subset in $CP$, then $A$ is a residual subset of $\CP$.
\end{corollary}

\begin{remark}
 Transitive interval maps have dense periodic points, thus all maps from $\overline{CP}\setminus CP$ are not transitive.
\end{remark}

\begin{lemma}\label{l5'}The set of piecewise affine leo maps is dense in $CP$.\end{lemma}
\begin{proof} Fix $f \in CP$, choose a homeomorphism $\psi$ of $I$ such that $\hat{f} := \psi \circ f \circ \psi^{-1} \in C_{\lambda}$ (see \cite[p. 2] {BCOT}).
Proposition 8 from \cite{BT} states that piecewise affine leo maps are dense in $C_{\lambda}$.
Choose a sequence $(\hat{f}_n)$ of piecewise affine maps in  $C_{\lambda}$ converging to $\hat{f}$
and choose a sequence of piecewise affine homeomorphism $\psi_n$ converging to $\psi$.
Let $f_n :=  \psi_n^{-1} \circ \hat{f}_n \circ \psi_n$. The maps $f_n$ are clearly piecewise
affine, they are leo  since the leo property is preserved by  conjugation by a homeomorphism. Since leo maps have dense periodic points $f_n \in CP$.  Finally 
$f_n = \psi_n^{-1} \circ \hat{f}_n \circ \psi_n \to  \psi^{-1} \circ \hat f \circ \psi
= f$ since $\psi_n \to \psi$ and $\hat f_n \to \hat f$.
\end{proof}

\begin{theorem}\label{thm:CPclosurevsCP}
$\overline{CP}\setminus CP$ is dense in $\overline{CP}$.
\end{theorem}

\begin{proof} 
Start with 
any piecewise affine topologically mixing map $T$ and $\eta \in (0,1/2)$. Note that any piecewise monotone topologically mixing map is leo (see Proposition 2.34 in \cite{R}). 
We will construct a non-transitive map with no flat parts in $B(T,\eta)$ which can be uniformly approximated by maps from $CP$.  Once we have constructed such a map, this completes the proof since piecewise affine leo maps are dense in $CP$ (Lemma \ref{l5'}).
More precisely we will perform a construction from \cite[Example~4.6]{KO} resembling a Denjoy map \cite[Example~14.9]{Devaney}.

First we choose a point $z\in (0,1)$ with a dense orbit under map $T$, which exists since $T$ is leo. Denote by $D_0:=\{z,T(z)\}\cup T^{-1}(\{z\})$ and inductively set
$D_{n+1}:=T(D_n)\cup T^{-1}(D_n)$. Let
$$
D:=\bigcup_{n=1}^\infty D_n.
$$
 Note that $D$ is a dense and countable set such that $T(D)=D$. Since the set of points with dense orbit is residual and the space is perfect, it is uncountable. Denote the set of turning points of $T$ by $\mathcal{T}_T$. Since $T$ is a piecewise affine map $\mathcal{T}_T$ is finite and hence we can replace $z$ by another point when necessary, ensuring this way that also $D\cap (\{0,1\}\cup \mathcal{T}_T)=\emptyset$.
Enumerate its elements by $D=\{x_i\; : \; i\in \N\}$ where $x_i\neq x_j$ for $i\neq j$.
Furthermore observe that if $T^n(x_i)=x_j$ for some $n>0$ then $i\neq j$ and $x_i\notin \{x_j,T(x_j),T^2(x_j),\ldots\}$, if not $z$ would be an eventually periodic point.
By definition, both $D$ and $[0,1]\setminus D$ are invariant under $T$. Define a function $\phi \colon \N \to \N$ so that $T(x_i)=x_{\phi(i)}$.

We follow the standard Denjoy construction, see \cite{Devaney} for details.
At the $i$-th step we remove the point $x_i$ from the interval and replace the hole with an interval $I_i$ of length  $\eta^{i+1}$ (the new interval is longer, so it 
naturally redefines the set $\{x_j : j> i\}$). This way a new continuous map $F$ is defined on the extended space so that each interval $I_i$ is mapped by an affine map 
 onto $I_{\phi(i)}$ in such a way that $F$ is continuous.
 
 Note that $F$ is semi-conjugate to $T$ by collapsing  the intervals $I_i$ back to single points as explained below.
As the domain of $F$ is isometric to $[0,1+\gamma]$, where $\gamma=\sum_{i\in \N}\eta^{i+1} < \eta$ we view $F\colon [0,1+\gamma]\to [0,1+\gamma]$. In this way every interval $I_i$ becomes some interval $[a_i,b_i]\subset (0,1+\gamma)$ and there is a quotient map $\pi\colon [0,1+\gamma]\to [0,1]$ that does not increase distance, collapses every interval $[a_i,b_i]$ into a single point $x_i$, and has the property that $T\circ \pi=\pi\circ F$. If we fix $i,j$, so that $x_i\not\in \{x_j,T(x_j),T^2(x_j),\ldots\}$ then $F^n((a_j,b_j))\cap (a_i,b_i)=\emptyset$ for all $n>0$.

Thus, the map $F$ is not transitive, and is non-constant on any open interval. 
Nonetheless 
we will show that it can be obtained as uniform limit of maps from $CP$.
To do so, let $F_n$ be the map obtained from $F$ in the following way:

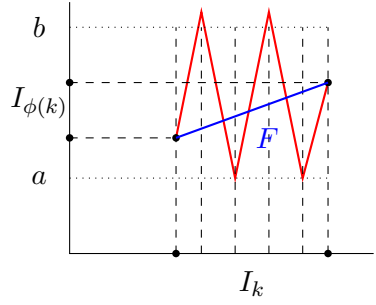
\begin{figure}
	\begin{tikzpicture}[scale=4]
	\draw (0,5/6)--(0,0)--(1,0);
	\draw[thick,red] (1/4+1/10,1/3+1/20)--(1/3+1/10,3/4+1/20)--(4/9+1/10,1/4)--(5/9+1/10,3/4+1/20)--(2/3+1/10,1/4)--(3/4+1/10,2/3-2/20);
	\draw[dashed] (1/4+1/10,0)--(1/4+1/10,3/4);
	\draw[dotted] (0,1/4)--(3/4+1/10,1/4);
	\draw[dotted] (0,3/4)--(3/4+1/10,3/4);
	\draw[dashed] (3/4+1/10,0)--(3/4+1/10,3/4);
	\draw[dashed] (0,1/3+1/20)--(1/3,1/3+1/20);
	\draw[dashed] (0,2/3-2/20)--(3/4+1/10,2/3-2/20);
	\draw[dashed] (1/3+1/10,0)--(1/3+1/10,3/4);
	\draw[dashed] (4/9+1/10,0)--(4/9+1/10,3/4);
	\draw[dashed] (5/9+1/10,0)--(5/9+1/10,3/4);
	\draw[dashed] (2/3+1/10,0)--(2/3+1/10,3/4);
	\node at (1/2+1/10,-0.1) {\small $I_k$};
	\node at (-0.1,1/2) {\small $I_{\phi(k)}$};
 \node at (-0.1,3/4) {\small $b$};
 \node at (-0.1,1/4) {\small $a$};
	\node[circle,fill, inner sep=1] at (1/4+1/10,0){};
	\node[circle,fill, inner sep=1] at (3/4+1/10,0){};
	\node[circle,fill, inner sep=1] at (0,1/3+1/20){};
	\node[circle,fill, inner sep=1] at (0,2/3-2/20){};
	\node[circle,fill, inner sep=1] at (1/4+1/10,1/3+1/20){};
	\node[circle,fill, inner sep=1] at (3/4+1/10,2/3-2/20){};
	\draw[thick,blue] (1/4+1/10,1/3+1/20)--(3/4+1/10,2/3-2/20);
	\node at (1/2+1/20+1/10,1/3+1/20) {\color{blue}\small $F$};
	\end{tikzpicture}
	\caption{Construction of $F_n$. Note that $b-a= \diam(I_k)$.}\label{fig:F_n}
\end{figure}

\begin{enumerate}
 \item outside of $\bigcup_{k\geq n}I_k$ we have $F_n(x)=F(x)$.

    \item  
     If $|k|\geq n$ then $F_n|_{I_k}$ is defined as the map with $5$ pieces of monotonicity,  constant slope, and a $3$-fold map constructed using the interior 3 pieces of monotonicity with the properties that on each piece of monotonicity $J$ we have $F_{n}(J)\supset I_{\phi(k)}=F(I_k)$ and 
    $$
    \gamma>\diam F_n(J)> \max\{\diam I_k, \diam I_{\phi(k)}\}.
    $$ 
    The remaining two pieces of the monotonicity
    with one endpoint contained in the set of endpoints of $I_k$ are used to make the map $F_n$ continuous.
    We can additionally require that the slope on each piece of $I_k$ is at least $3$ because the pieces of monotonicity do not need to be of equal length.
    
\end{enumerate}
It is clear that $F_n$ converges uniformly to $F$. We claim that each $F_n$ is topologically mixing. Indeed, fix any interval $J$ and any $\eps>0$.
Note that if an interval $I\subset I_k$ for some $|k|\geq n$ and $F_n(I)\subset I_{\phi(k)}$ and $|\phi(k)|\geq n$ then either $\diam F_n(I)>\diam(I)$
or $F_n(I)$ contains two consecutive critical points in $I_{\phi(k)}$ and then $F_n^2(I)\supsetneq I_{\phi^2(k)}$. But since diameter of an interval can not grow
to infinity, it means that for $J$ there is some iterate $m$ such that $F_n^m(J)\not\subset I_k$ for any $k$. But since $T$ is mixing and $\pi(F_n^m(J))$
is a nondegenerate interval, for any $\delta>0$ there are points $u,v\in \pi(F_n^m(J))$ and $j>0$ such that $T^j(u)< \delta$ and $T^j(v)>1-\delta$.
But since $\delta$ is arbitrarily small, there are $p,q\not\in \bigcup_{k}I_k\cap F_n^m(J)$ such that $F_n^j(p)<\eps$ and $F_n^j(q)>1+\gamma-\eps$. This means
that $[\eps,1+\gamma-\eps]\subset F_n^{m+j}(J)$, because $F^j(p)=F^j_n(p)$ and $F^j(q)=F^j_n(q)$ showing that $F_n$ is mixing. This proves the claim, and as a consequence each $F_n$ has dense periodic points.

To finish the proof, let $\zeta\colon [0,1]\to [0,1+\gamma]$ be a linear onto map with $\zeta(0)=0$
and let $G=\zeta^{-1}\circ F\circ \zeta$.
Note that if we fix any $z\in [0,1+\gamma]$ then there is $z'\in [0,1+\gamma]$ such that $z'$ is either endpoint of $I_k$ for some $k$ or $z'=z$
and $z\not\in I_k$ for any $k$.
In any of these cases $T(\pi(z))\leq F(z')\leq T(\pi(z))+\gamma$.
But $|F(z')-F(z)|<\sup_{k}\diam F(I_k)\leq \gamma$. This means that 
$T(\pi(z))-2\gamma\leq F(z)\leq T(\pi(z))+2\gamma$.
On the other hand, for any $x\in [0,1]$
there is $z\in [0,1+\gamma]$ such that $\pi(z)=x$ and so
$T(x)-2\gamma\leq F(z)\leq T(x)+2\gamma$.
Fix any $x\in [0,1]$ and let $z=\zeta(x)$ and $y=\pi(z)$. Note that $d(x,z)<\gamma$
and $d(y,z)\leq \gamma$ so $d(x,z)<\gamma$
This gives
\begin{eqnarray*}
d(T(x),G(x))&=&d(T(x),\zeta^{-1}(T(x))) + d(\zeta^{-1}(T(x)),\zeta^{-1}(F(z))) \leq
\\
&\leq& \gamma+d(T(x),F(z))\leq 3\gamma+d(T(x),T(y)).
\end{eqnarray*}
But if $\gamma\to 0$ then $d(T(x),T(y))\to 0$, hence taking small $\gamma$
we can make $\rho(T,G)$ arbitrarily small.
This completes the proof.
\end{proof}

\begin{proposition}
Sets $\CP$ and $CP$ are nowhere dense in $C(I)$.
\end{proposition}

\begin{proof}
First note that the map $x\mapsto x^2\notin \CP$ and the same is true for any map in its small neighbourhood in $C(I)$. For what follows we refer the reader to the proof of Theorem~\ref{thm:CPclosurevsCP} where a procedure of blowing up the points is described in more detail. Fix $\eps>0$, $f\in CP$ and take any $g\in C(I)$ that is topologically mixing and piecewise affine, has finite set of local extrema (i.e., there is no interval with constant value) and $\eps$-close to $f$. Let $p$ be a fixed point of $g$ and $D=\cup_{n\geq 0}g^{-n}(p)$. Observe that $D$ is countable and dense in $I$. We enumerate it  $D=\{a_0 = p,a_1,\ldots\}$. Now take a collection of intervals $\{I_n\}_{n\geq 0}$ such that $\sum_{n\geq0}\lambda(I_n)=\gamma$ where $\gamma>0$ is small.
As in the proof of Theorem~\ref{thm:CPclosurevsCP} we modify $g$ to a continuous map $G$ by blowing up each point $a_n$ to an interval $I_n$, and 
if $g(a_n)=a_m$ ($n > 0$) then we define $I_m=G(I_n)$ putting a piecewise affine map properly assigning the values of endpoints so that continuity and monotonicity are preserved; and $G: I_0 \to I_0$ is a  properly rescaled version of   $x \mapsto x^2$, i.e., rescaled so that $G(I_0) = I_0$.
Since $\bigcup I_m$ is dense in $[0,1+\gamma]$
the map $G$ is uniquely defined on
its complement.
Interior points of $I_0$ cannot be periodic points of $G$, thus any map sufficiently close to $G$ can not be in $CP$. Furthermore, if
we take $\eps \to 0$ and $\gamma \to 0$ 
the resulting maps $G$ converge to $f$.
\end{proof}

\subsection{$CP$ and $\overline{CP}$ are uniformly locally arcwise connected}\label{subsec:locallyconnected}

The first step is to prove the following theorem that answers a question  Micha{\l}  Misiurewicz asked at during a talk at the ``Workshop on topological and combinatorial dynamics'' in April 2021 at  
at the CRM in Montreal.

Recall that $f,g\in C(I)$ are homotopic, if there is a family of maps $\{h_t\}_{t\in [0,1]}$ such that $h_0:=f$, $h_1=g$ and the map $(x,t)\mapsto h_t(x)$ is continuous from $I\times[0,1]$ to $I$.

\begin{theorem}\label{thm:ac}
Every map $f\in C_{\lambda}$ is homotopic to $\mathrm{Id}$.
\end{theorem}
\begin{proof}
We first show how to reduce  to the case when $f(0)=0$. 
 Suppose $f(0)>0$ and define $\eps:=\inf \{x \in [0,1]: f(x)=0\}$. 
For $a\in [0,\eps]$ we define $f_a$ as a two-fold perturbation on the interval $[0,a]$. $\{f_a\}_{a\in [0,\eps]}$ is a continuous family of maps in $C_{\lambda}$ and $f_{\eps}(0) \in \{0,1\}$. Let $g :=f_\eps$. 
Next, for every $\alpha\in [0,1]$ we define $g_{\alpha}:I\to I$ by: 
$$
g_{\alpha}(x):=\begin{cases}
x,& x\leq \alpha,\\
\alpha + (1-\alpha)g(\frac{x-\alpha}{1-\alpha}), & x>\alpha.
\end{cases}
$$
We have $g_{\alpha}\in C_{\lambda}$ for all $\alpha\in[0,1]$, and they form a continuous family of maps and $g_1=\textrm{Id}$.
\end{proof}

\begin{corollary}\label{cor:ac}
The space $C_{\lambda}$ is arcwise connected.
\end{corollary}

A metric space $(X,d)$ is called {\it uniformly locally arcwise connected} if for any $\eps>0$ there is a $\delta>0$ such that whenever $0<d(x,y)<\delta$, then $x$ and $y$ are joined by an arc of diameter smaller than $\eps$ (see \cite{HY}). For the following definitions and statements see \cite[Subsections 5.10 and 5.22]{Nadler}. A metric space $(X,d)$ is {\em connected im kleinen at $x\in X$} if for each open set $x\in U$, there is a component of $U$ which contains $x$ in its interior. The space $X$ is {\em locally connected at $x$} if for each open set $U$ containing $x$ there is a connected open set $V$ such that $x\in V\subset U$. We say a space is {\em locally connected} if it is locally connected at every $x\in X$. Note that connectedness im kleinen at every $x\in X$ implies local connectedness. However, uniformly locally arcwise connectedness implies connectedness im kleinem at every $x\in X$ since such a point will be always contained in the interior of some arc. Thus uniformly locally arcwise connectedness implies local connectedness. In \cite{KMS} it is proven that 
each of the three spaces: all transitive maps, all piecewise monotone transitive maps and all piecewise linear transitive maps are uniformly locally arcwise connected. 

The fact that transitive maps have dense sets of periodic points allows us to restate Lemma~2.6 from \cite{KMS}, which will help us to prove arcwise connectedness of $\CP$.
\begin{lemma}\label{lem:ulac}
For any transitive maps $f, g\in CP$ there is an arc $A\subset CP$ joining $f$ and $g$ such that its diameter is smaller than or
equal to $5d(f,g)$.
\end{lemma}

\begin{theorem}\label{thm1}
$CP$ and $\overline{CP}$ are uniformly locally arcwise connected.
\end{theorem}

\begin{proof}
Fix any maps $f,g\in CP$. 
Let $\delta=d(f,g)$.
Since by Lemma~\ref{l5'} leo maps are dense in $CP$, there are sequences $f_n\to f$ and $g_n\to g$ of leo maps.
We may also assume that $d(g_n,g)<\delta/2^{n+1}$, so in particular $d(g_n,g_{n+1})<\delta/2^{n+1}$ and $d(g_0,f_0)<2\delta$.
Let $A_0$ be the arc joining $f_0,g_0$,
$A_n$ arc joining $g_{n-1},g_n$ and $A_{-n}$ an arc joining $f_{n-1},f_n$.
Furthermore, we assume that these arcs are obtained by application of Lemma~\ref{lem:ulac}, so $\diam A_0<10\delta$ and $\diam A_n\leq 5\delta/2^{|n|}$. It is also clear that if we denote $J_n=\cup_{i=-n}^n A_n$ then it is an arc, and $d(J_n,J_m)<5\delta/2^n$ for any $m>n$, so it is a Cauchy sequence.
Then the limit $J=\lim_n J_n$ exists. But it is also clear that $\lim_n \diam (\sup_{i\geq n}A_i)=0$ and $\lim_n \diam (\sup_{i\geq n}A_{-i})=0$, so $\cup_n J_n$ compactifies with two points, which in fact are $f,g$ as limits of endpoints of $J_n$. This shows $J$ is an arc and also that $J\subset CP$. By our construction
$$
\diam J\leq \sum_{n=0}^\infty 10\delta /2^{n}\leq 20\delta.
$$
This proves uniform local arcwise connectedness of $CP$.

The proof that $\overline{CP}$ is uniformly locally arcwise connected follows the same lines as for $CP$ with the only difference that in place of the sequence of leo maps we now can use elements of $CP$
 and its uniform local arcwise connectedness.
 \end{proof}

We do not know if an analog of Lemma~\ref{lem:ulac} is true for $C_\lambda$. More precisely, it would be interesting to know the answer of the following question.
Fix any $f, g\in C_\lambda$. Is there a constant $\eta>0$ and an arc $A\subset C_\lambda$ connecting $f$ with $g$ 
such that $\diam A\leq \eta d(f,g)$?

\section{Denseness properties in $CP$ and $\overline{CP}$}\label{sec:densness}

\subsection{Window perturbations in CP}\label{sec-win-pert}\label{subsec:windowperturbations} The proofs in this section use window perturbations as a tool.
Let $J\subset I$ be an interval, $m$ an odd positive integer and $\{J_i \in I: 1 \le i \le m \}$ a finite collection of intervals satisfying $\cup^{m}_{i=1} J_i = J$ and $\Int(J_i) \cap \Int(J_j) = \emptyset $ when $i \ne j$. We will refer to this as a {\it partition} of $J$.
 
In our previous articles we considered the following notion in $C_{\lambda}$. Fix $f \in C_{\lambda}$
 an interval $J \subset I$ and a partition $\{J_i\}$ of $J$.
 A map $g \in C_{\lambda}$ is {\it an $m$-fold window perturbation of $f$ with respect to $J$} and the partition $\{J_i\}$ if
 \begin{itemize}
     \item $g|_{J^c} = f|_{J^c}$ 
     \item for each $1 \le i \le m$ the map $g|_{J_i}$ is an affinely scaled copy of $f|_{J}$  with  $g|_{J_1}$ having the same orientation as $f|_J$, $g|_{J_2}$ having the opposite orientation to $f|_{J}$, and then continuing with the orientations alternating.
 \end{itemize}
 
The essence of this definition is illustrated by Figure~\ref{fig:perturb}.

We will apply this in $CP$ in the following way.  We fix $f \in CP$
and conjugate $f$ to a map $\hat{f} = \psi^{-1} \circ f \circ \psi \in C_{\lambda}$.  We can apply the above procedure to $\hat{f}$ to obtain an $m$-fold window perturbation $\hat{g}$ of $\hat{f}$, and then we call the map $g := \psi \circ \hat{g} \circ \psi^{-1}$ an $m$-fold window perturbation of $f$.

We call an $m$-fold wind perturbation regular
if the intervals $\psi(J_i)$ all have the same length.
 
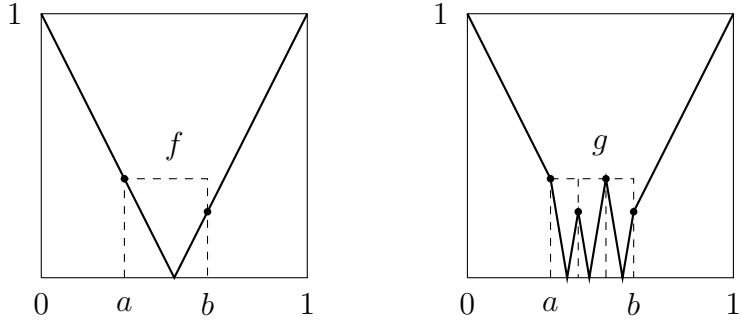
\begin{figure}[!ht]
	\centering
	\begin{tikzpicture}[scale=3.5]
	\draw (0,0)--(0,1)--(1,1)--(1,0)--(0,0);
	\draw[thick] (0,1)--(1/2,0)--(1,1);
	\node at (1/2,1/2) {$f$};
	\node at (5/16,-0.1) {$a$};
	\node at (5/8,-0.1) {$b$};
	\node at (0,-0.1) {$0$};
	\node at (1,-0.1) {$1$};
	\node at (-0.1,1) {$1$};
	\draw[dashed] (5/16,0)--(5/16,3/8)--(5/8,3/8)--(5/8,0);
	\node[circle,fill, inner sep=1] at (5/16,3/8){};
	\node[circle,fill, inner sep=1] at (5/8,1/4){};
	\end{tikzpicture}
	\hspace{1cm}
	\begin{tikzpicture}[scale=3.5]
	\draw (0,0)--(0,1)--(1,1)--(1,0)--(0,0);
	\draw[thick] (0,1)--(5/16,3/8)--(5/16+3/48,0)--(5/16+5/48,1/4)--(5/16+7/48,0)--(5/16+10/48,3/8)--(5/16+13/48,0)--(5/8,1/4)--(1,1);
	\draw[dashed] (5/16,0)--(5/16,3/8);
	\draw[dashed] (5/16+10/48,0)--(5/16+10/48,3/8);
	\draw[dashed] (5/16+5/48,3/8)--(5/16+5/48,0);
	\node at (1/2,1/2) {$g$};
	\node at (0,-0.1) {$0$};
	\node at (1,-0.1) {$1$};
	\node at (-0.1,1) {$1$};
	\node at (5/16,-0.1) {$a$};
	\node at (5/8,-0.1) {$b$};
	\draw[dashed] (5/16,3/8)--(5/8,3/8)--(5/8,0);
	\node[circle,fill, inner sep=1] at (5/16,3/8){};
	\node[circle,fill, inner sep=1] at (5/16+5/48,1/4){};
	\node[circle,fill, inner sep=1] at (5/16+10/48,3/8){};
	\node[circle,fill, inner sep=1] at (5/8,1/4){};
	\end{tikzpicture}
	\caption{For  $f\in C_{\lambda}$ shown on the left, we show on the right the graph of $g$ which is
 a $3$-fold piecewise window perturbation of $f$ on the interval $[a,b]$.}\label{fig:perturb} 
\end{figure}

\subsection{Leo maps form an open dense collection in $CP$}\label{subsec:leo}

We state a part of the Barge-Martin result \cite{BM} which we use later.

\begin{proposition}\label{prop:BM} Suppose $f \in CP$. The following  assertions hold.
\begin{itemize}[(a)]
\item[(a)] There is a collection (perhaps finite or empty) $\J= \J(f) =\{J_1,J_2,\dots\}$ of
closed subintervals of $[0,1]$ with mutually disjoint interiors, such that for each $i$,  $f^2(J_i) = J_i$, and
there is a point $x_i\in J_i$ such that $\{f^{4n}(x_i)\colon~n\ge 0\}$ is dense in $J_i$.
\item[(b)] If $x\in (0,1)\setminus \bigcup_{i\ge 1}\mathring{J}_i$, then $f^2(x)=x$.
\item[(c)] For each $J\in \J$, the map $f^2|_J$ is topologically mixing.
\end{itemize}
 \end{proposition}

At the first glance, one might naively expect that a similar characterization holds in $\CP$,
namely, for a map in $\CP \setminus CP$ one could hope to partition $I$ based on restricting to a different $\omega$-limit sets. The following example shows this is not always possible.
First note that the map $f$ from Figure~\ref{fig:nodecomp} is in $\CP\setminus CP$. Indeed, $f$ is not in $CP$ since periodic points are not dense in $[1-\eps,1]$ for $\eps>0$ small. However $f\in \CP$ since one can approximate the 
leftmost critical value of $f$ with maps that have critical values slightly above $3/4$ and thus these perturbations will all be transitive and therefore in $CP$. On the other hand, we cannot expect a Barge-Martin decomposition for 
$f$. Namely, the set $[0,3/4]$ is $f$-invariant, and there exists a Cantor set of points in $[3/4,1-\eps]$, for small $\eps>0$, which get mapped to $[0,3/4]$. Therefore, we 
cannot separate the $\omega$-limit set of  $f|_{[0,3/4]}$ from the $\omega$-limit set of $f|_{[3/4,1]}$. 

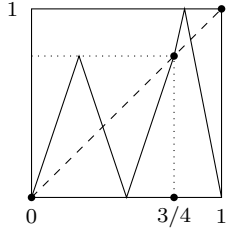
\begin{figure}[!ht]\label{fig3}
	\centering
	\begin{tikzpicture}[scale=2.5]
	\draw (0,0)--(0,1)--(1,1)--(1,0)--(0,0);
	\draw[dotted] (0,3/4)--(3/4,3/4)--(3/4,0);
	\draw[dashed] (0,0)--(1,1);
	\draw(0,0)--(3/12,3/4)--(1/2,0)--(3/4,3/4)--(3/4+1/18,1)--(1,0);
	\node[circle,fill, inner sep=1] at (3/4,0){};
	\node[circle,fill, inner sep=1] at (3/4,3/4){};
	\node[circle,fill, inner sep=1] at (0,0){};
	\node[circle,fill, inner sep=1] at (1,1){};
	\node at (0,-0.1) {\tiny $0$};
	\node at (1,-0.1) {\tiny $1$};
	\node at (0.75,-0.1) {\tiny $3/4$};
	\node at (-0.1,1) {\tiny $1$};
	\end{tikzpicture}
\caption{An example explaining that Barge-Martin-like decomposition in $\CP$ is not possible.}
\label{fig:nodecomp}
\end{figure}

The following definition is crucial in the proof of Theorem \ref{thm:leo}.

\begin{definition}For $f\in CP$, an interval or a point $J\subsetneq [0,1]$ \emph{$f$-splits $I$} if 
\begin{itemize}
\item either $J=f(J)$ 
\item or $J$ and $f(J)$ have disjoint interiors and $f^2(J)=J$
\end{itemize}
and each nonempty connected component $C$ of $I\setminus (J\cup f(J))$ satisfies $f^2(\overline{C})=\overline{C}$.
\end{definition}

\begin{theorem}\label{thm:leo} Leo maps form an open dense collection in $CP$.
\end{theorem}

\begin{proof}Let 
\begin{equation*}
M :=\{f\in CP\colon~f(\{0,1\})\cap \{0,1\}=\emptyset\}.
\end{equation*}
Clearly $M$ is open in $CP$.  We claim $M$ is also dense in $CP$. Fix $f\in CP$ and $\eps > 0$. By Remark from \cite{BCOT} there is a non-atomic invariant measure $\mu$ for $f$ with full support on $I$ and a homeomorphism $\psi:I\to I$ such that $g:=\psi^{-1}\circ f\circ \psi\in C_{\lambda}$. If $g(0)=0$ or $g(0)=1$ then we can perform a $2$-fold window perturbation of size at most $\eps$ on the arc $[0,\eps]$ to destroy the fixed points and similarly we can destroy $f(1)=0$ or $f(0)=1$. We obtain a map $G\in C_{\lambda}$ with $\rho(G,g) < \eps$ such that $G(0),G(1)\notin \{0,1\}$. Note that $G \in CP$, and thus the map
$F:=\psi\circ G\circ \psi^{-1}\in CP$ as well. 
Since this is true for every $\eps$ there exists a map $F \in CP$ arbitrarily close to $f$ with $F(0),F(1)\notin \{0,1\}$.

  Notice that if $f \in M$ then $f^2\ne\mathrm{Id}$. 
	If $f^2\ne \mathrm{Id}$, the set of intervals $\J$ in Proposition \ref{prop:BM} is non-empty.

We claim that the set
\begin{align*}N:=&\{g\in M\colon~\text{a periodic interval or point of period $p\in\{1,2\}$ $g$-splits }[0,1] \}
\end{align*}
coincides with the set of maps in $M$ that are not leo. 
The only if direction is clear.

 Suppose first that $f \in M$ is not leo and that 
$\#\, \J =1$. Note that such an $f$ can not exists in $M$, but we will not prove or use this fact directly. We claim that  
$J \subsetneq I$. Suppose the contrary, i.e., $J = I$.  By Proposition~\ref{prop:BM}{\it (c)} $f$ is topologically mixing, but for any interval $(a,b)$ we have
$f^n(a,b)$ is an interval and $\lambda(f^n(a,b)) \to 1$.  Since $f \in CP$ it is surjective and since 
$f \in M$ we have some $0<c<d<1$ such that $f(c) = 0$ and $f(d) =1$ (or vice versa). Combining these two fact yields that $f^n(a,b) = I$ for sufficiently large $n$, contradicting that $f$ is not leo.  We have shown $J \subsetneq I.$
By Proposition~\ref{prop:BM}{\it (b)} $f^2(\overline{C}) = \overline{C}$ for each connected component $C$ of $J^c$ and thus $J$ $f$-splits $I$.

 Now suppose $f \in M$ is not leo and $J_1$, $J_2 \in \J$  then again by 
Proposition~\ref{prop:BM}{\it (b)}
$f^2(\overline{C}) = \overline{C}$ for each nonempty connected component $C$ of $(J_1 \cup J_2)^c$ and thus either of the $J_i$ $f$-splits $I$. This finishes the proof of the claim.

 Let us prove that $N$ is closed in $M$.  Suppose that $\{f_n\} \subset N$ converges to $f \in M$. Since $I$ is compact we can choose a subsequence of $n_i$'s such that the corresponding intervals $J_{n_i}$ converge to an interval $J$, which is  possibly degenerate to a point or $J = [0,1]$. Since 
$f_n^2(J_n) = J_n$
and the $f_n$ are continuous we get $f^2(J) = J$.

If $J\neq [0,1]$ and $J \cap (0,1) \ne \emptyset$ then $J$ $f$-splits $I$ and so $f \in N$.
If $J = \{0\}$
then $f(0) = 0 $ or $f(0) =1$ and so $f \not \in M$, and similarly if $J = \{1\}$ we have $f \not \in M$.

Finally we turn to the case $J = [0,1]$. But in this case it follows from Proposition~\ref{prop:BM}(a) and the fact $f_n\to f$ that $f(0)=0$ or $f(0)=1$ (and similarly for $f(1)$). Therefore, also in this case $f\notin M$.

To conclude the proof note that the set $N$ is nowhere dense in $M$ since the set of piecewise affine leo maps is dense in $M$ by Lemma~\ref{l5'}.
\end{proof}

\begin{remark}\label{rem:leolambda}
The set of leo piecewise affine interval maps $f$ preserving $\lambda$ such that $f(\{0,1\})\cap \{0,1\}=\emptyset$ forms a dense collection in $C_{\lambda}$, see \cite[Proposition 8]{BT}.
Thus the proof of Theorem~\ref{thm:leo} works without change in the setting of $C_{\lambda}$.
Suppose $\mu$ is a non-atomic invariant measure $\mu$ with full support.
By Remark~\ref{rem:psihomeo} leo maps are open and dense in $C_{\mu}$ as well. 
\end{remark}

\begin{remark}
In fact the proof of Theorem~\ref{thm:leo} shows more, namely that the set of leo maps $f$ such that $f(\{0,1\})\cap \{0,1\}=\emptyset$ is an open and dense subset of $CP$. Suppose $\mu$ is a non-atomic invariant measure $\mu$ with full support. By Remark~\ref{rem:leolambda} the space $C_{\mu}$ contains an open dense set of leo maps $f$ such that $f(\{0,1\})\cap \{0,1\}=\emptyset$.  
\end{remark}

 Applying Corollary \ref{cor:resinCPresinCPclosure} and Theorem \ref{thm:CPclosurevsCP} yields

\begin{corollary}
The set of leo maps is residual in $\CP$  but is not open.
\end{corollary}

Blokh showed
that topological mixing implies the periodic specification property \cite[Theorem 3.4]{R} (see also \cite[Appendix A]{Buzzi}), thus we have

\begin{corollary}\label{cor:psp}
Maps satisfying the periodic specification property form an open dense collection in $CP$, $C_{\lambda}$ and $C_{\mu}$.
\end{corollary}

\subsection{Conjugacy classes in $\overline{CP}$}\label{subsec:conjugacyclasses} 

A beautiful result of Kechris and Rosendal \cite{KR} shows that there is residual set in the space of all homeomorphisms of the Cantor set such that each pair of its elements are conjugate (so-called \textit{residual conjugacy class}). Moreover, Bernardes and Darji \cite{BD} proved there is a residual conjugacy class in the space of all self-maps of the Cantor set. In the case of $\CP$ (or  $C_{\mu}$ and $C(I)$) all the standard invariants (entropy, mixing, ... ) do not exclude the possibility that such a class exists. 
Our next result shows that in $\CP$ a residual conjugacy class  does not exist.   

Let $\mathcal{H}(I)$ denote the set of all homeomorphism (increasing  or decreasing) of $I$ and for $f\in \CP$ put $$G_f :=\{\psi^{-1} \circ f \circ \psi: \psi\in\mathcal{H}(I)\}.$$

 Denote by $\mathcal{R}:=\bigcup_{f\in \CP}\{f\}\times G_{f}$ the {\em conjugacy relation} in $\CP\times \CP$.

Theorem \ref{ChainRecurrentThenCPClosure} directly implies that the set $\overline{G_f}$ is a subset of $\CP$. 

As usually, $\sgn\colon~\mathbb R\to\{-1,0,1\}$ is defined as 
$$
\sgn(x):=\begin{cases}
-1,& x<0,\\
0,& x=0,\\
1,& x>0.
\end{cases}
$$

 \begin{theorem}\label{conj-nowhere}
For every $f\in \CP$, the set $G_f$ is nowhere dense in $\CP$.
\end{theorem}

\begin{proof}Let us fix $f\in\CP$. For every $\psi\in \mathcal{H}(I)$ define $f_\psi:=\psi^{-1} \circ f \circ \psi\in \CP$. We start with three basic but useful observations:
\begin{itemize}
    \item[(A)] $\forall~\psi\in\mathcal{H}(I)$, $\sgn(f(0)-f(1))=\sgn(f_\psi(0)-f_\psi(1))$.  
    \item[(B)] If $f_{\psi_n}\rightrightarrows g$ then $f^m_{\psi_n}\rightrightarrows g^m$ for each $m\in\mathbb N$ (by induction  over $m$).
    \item[(C)] For each $\psi\in\mathcal{H}(I)$ the map $H_\psi\colon~ \CP\to \CP$ defined as $H_\psi(f)=\psi^{-1}\circ f\circ \psi$, $f\in\CP$, is a homeomorphism of $\CP$. 
 \end{itemize}
 Assume to the contrary that there exists a nonempty open set $U\subset\overline{G}_f\subset \CP$.  
  Notice that by (C) $H_{\psi}(U)$ is an open set in $\CP$.
 Using \cite[Remark p.2]{BCOT} there exist a homeomorphism $\psi \in \mathcal{H}(I)$ and a piecewise affine leo map $\alpha$ such that $\alpha\in H_\psi(U)\cap C_{\lambda}$ and each of the points $0$ and $1$ have a preimage outside the set $\{0,1\}$ using Remark~\ref{rem:leolambda}. Fix $\eps>0$ so small that $B(\alpha,\eps)\subset  H_{\psi}(U)$. There is a small $\delta >0$ such that each window perturbation of $\alpha$ on the intervals $[0,\delta]$, $[1-\delta,1]$ remains in $B(\alpha,\eps)$,  see Subsection~\ref{subsec:windowperturbations}.  For $u\in [0,1]$ and $n\in\mathbb N$ denote $$O_{\alpha}^{-n}(u):=\{x\in I\colon~ \alpha^{n}(x)=u\}.$$ 

Since $0$ and $1$ have  $\alpha$-preimages outside the set $\{0,1\}$, for each positive integer $n$ there is $x\in O_{\alpha}^{-n}(u)$, $u\in\{0,1\}$, such that $f^i(x)\notin \{0,1\}$ for $i=0,1,\ldots,n-1$.
Because $\alpha$ is leo, there exists $n$ so large that one can  choose $x_i,y_i$  for $i\in \{0,1\}$ such that
\begin{eqnarray*}
    x_0 &\in & O_{\alpha}^{-n}(0)\cap (0,\delta),\\
    x_1 &\in & O_{\alpha}^{-n}(1)\cap (0,\delta),\\
    y_0 &\in & O_{\alpha}^{-n}(0)\cap (1-\delta,1),\\
    y_1 &\in & O_{\alpha}^{-n}(1)\cap (1-\delta,1)
    \end{eqnarray*}
and define  
\begin{eqnarray*}
    S_0 &:=& \{s^0_1<s^0_2\cdots<s^0_{k_0}\}=\{x_0,\alpha(x_0),\dots,\alpha^{n-1}(x_0)\}\cap ([0,\delta]\cup [1-\delta,1]),\\
    S_1 &:= & \{s^1_1<s^1_2\cdots<s^1_{k_1}\}=\{x_1,\alpha(x_1),\dots,\alpha^{n-1}(x_1)\}\cap ([0,\delta]\cup [1-\delta,1]),\\
    T_0 &:= &\{t^0_1<t^0_2\cdots<t^0_{\ell_{0}}\}=\{y_0,\alpha(y_0),\dots,\alpha^{n-1}(y_0)\}\cap ([0,\delta]\cup [1-\delta,1]),\\
  T_1 & := & \{t^1_1<t^1_2\cdots<t^1_{\ell_{1}}\}=\{y_1,\alpha(y_1),\dots,\alpha^{n-1}(y_1)\}\cap ([0,\delta]\cup [1-\delta,1]).
    \end{eqnarray*}

As mentioned above these four sets are disjoint from the set $\{0,1\}$.

 Now we distinguish two cases:
    
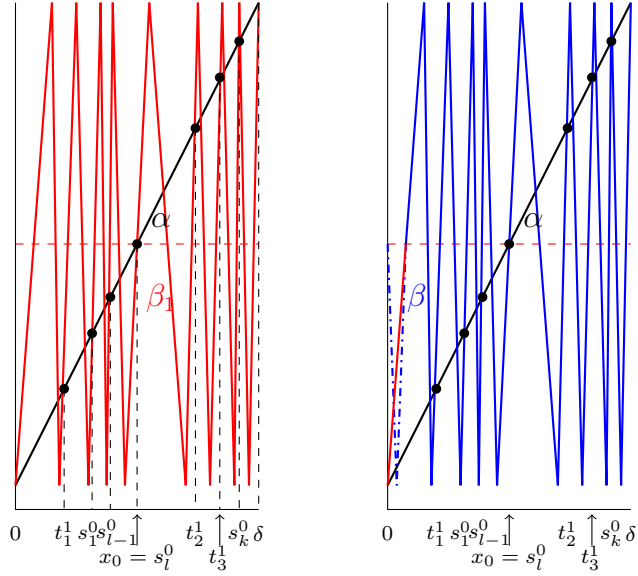
\begin{figure}
	\begin{tikzpicture}[scale=3.2]
	\draw (0,2.1)--(0,0)--(1,0);
	\draw[thick] (0,0.1)--(1,2.1);
	\draw[dashed,red] (0,1.1)--(1,1.1);
	\draw[thick, red] (0,0.1)--(0.15,2.1)--(0.18,0.1)--(0.25,2.1)--(0.3,0.1)--(0.35,2.1)--(0.375,0.1)--(0.8/2,2.1)--(0.9/2,0.1)--(0.5+0.05,2.1)--(0.5+0.2,0.1)--(0.5+0.25,2.1)--(0.5+0.3,0.1)--(0.5+0.35,2.1)--(0.5+0.4,0.1)--(0.92,2.1)--(0.96,0.1)--(1,2.1); 
	\node[circle,fill, inner sep=1.3] at (0.2,0.5){};
	\node[circle,fill, inner sep=1.3] at (0.315,0.73){};
	\node[circle,fill, inner sep=1.3] at (0.39,0.88){};
	\node[circle,fill, inner sep=1.3] at (0.5,1.1){};
	\node[circle,fill, inner sep=1.3] at (0.74,1.58){};
	\node[circle,fill, inner sep=1.3] at (0.84,1.79){};
	\node[circle,fill, inner sep=1.3] at (0.92,1.94){};
	\draw[dashed] (0.2,0.5)--(0.2,0);
	\draw[dashed] (0.315,0.73)--(0.315,0);
	\draw[dashed] (0.39,0.88)--(0.39,0);
	\draw[dashed] (0.5,1.1)--(0.5,0);
	\draw[dashed] (0.74,1.58)--(0.74,0);
	\draw[dashed] (0.84,1.79)--(0.84,0);
	\draw[dashed](0.92,1.95)--(0.92,0);
	\draw[dashed](1,2.1)--(1,0);
	\node at (0,-0.1) {\tiny $0$};
	\node at (0.2,-0.1) {\tiny $t^1_1$};
	\node at (0.3,-0.1) {\tiny $s^0_1$};
	\node at (0.42,-0.1) {\tiny $s^0_{l-1}$};
	\node at (0.5,-0.2) {\tiny $x_0=s^0_{l}$};
	\draw[->] (0.5,-0.15)--(0.5,-0.02);
	\node at (0.74,-0.1) {\tiny $t^1_{2}$};
	\node at (0.835,-0.2) {\tiny$t^1_{3}$};
	\draw[->] (0.84,-0.15)--(0.84,-0.02);
	\node at (0.92,-0.1) {\tiny$s^0_{k}$};
	\node at (1,-0.1) {\tiny $\delta$};
	\node at (0.6,7/8) {\red $\beta_1$};
	\node at (0.6,1.2) {$\alpha$};
\end{tikzpicture}
\hspace{1cm}
 \begin{tikzpicture}[scale=3.2]
	\draw (0,2.1)--(0,0)--(1,0);
	\draw[thick] (0,0.1)--(1,2.1);
	\draw[dashed,red] (0,1.1)--(1,1.1);
    \draw[ thick, blue, dash dot] (0,1.1)--(0.0375,0.1)--(0.075,1.1);
	\draw[thick, blue](0.075,1.1)--(0.15,2.1)--(0.18,0.1)--(0.25,2.1)--(0.3,0.1)--(0.35,2.1)--(0.375,0.1)--(0.8/2,2.1)--(0.9/2,0.1)--(0.5+0.05,2.1)--(0.5+0.2,0.1)--(0.5+0.25,2.1)--(0.5+0.3,0.1)--(0.5+0.35,2.1)--(0.5+0.4,0.1)--(0.92,2.1)--(0.96,0.1)--(1,2.1);
    \draw[thick,red] (0,0.1)--(0.075,1.1);
	\node[circle,fill, inner sep=1.3] at (0.2,0.5){};
	\node[circle,fill, inner sep=1.3] at (0.315,0.73){};
	\node[circle,fill, inner sep=1.3] at (0.39,0.88){};
	\node[circle,fill, inner sep=1.3] at (0.5,1.1){};
	\node[circle,fill, inner sep=1.3] at (0.74,1.58){};
	\node[circle,fill, inner sep=1.3] at (0.84,1.79){};
	\node[circle,fill, inner sep=1.3] at (0.92,1.94){};
	\node at (0,-0.1) {\tiny $0$};
	\node at (0.2,-0.1) {\tiny $t^1_1$};
	\node at (0.3,-0.1) {\tiny $s^0_1$};
	\node at (0.42,-0.1) {\tiny $s^0_{l-1}$};
	\node at (0.5,-0.2) {\tiny $x_0=s^0_{l}$};
	\draw[->] (0.5,-0.15)--(0.5,-0.02);
	\node at (0.74,-0.1) {\tiny $t^1_{2}$};
	\node at (0.835,-0.2) {\tiny$t^1_{3}$};
	\draw[->] (0.84,-0.15)--(0.84,-0.02);
	\node at (0.92,-0.1) {\tiny$s^0_{k}$};
	\node at (1,-0.1) {\tiny $\delta$};
	\node at (0.12,7/8) {\blue $\beta$};
	\node at (0.6,1.2) {$\alpha$};
\end{tikzpicture}
	\caption{\label{fig:conjugacyclasses} The maps $\alpha$ (black) and $\beta$ constructed in two steps ($\beta_1$ in red, then modified to $\beta$ in blue) from the proof of Theorem \ref{conj-nowhere} on the interval $[0,\delta]$ with $x_0=s^0_{\ell}$. On the interval $[{1-\delta,1}]$ the redefinition of $S_0$ and $T_1$ is analogous.}
\end{figure} 

\noindent  {\bf{I.}} $\sgn(f^n(0)-f^n(1))\in\{1,0\}$. As it is shown in Figure~\ref{fig:conjugacyclasses} in this case there exists a window perturbation $\beta$ of $\alpha$ on the intervals $[0,\delta]$, $[1-\delta,1]$ such that $\alpha$ and $\beta$ coincide on $S_0$ and $T_1$, hence 
\begin{eqnarray*}
    &\beta^i(0)=\alpha^i(x_0)
    & \text{ for each } i\in\{1,\dots,n\} \text{ and also } \\
    &\beta^i(1)=\alpha^i(y_1)
    & \text{ for each } i\in\{1,\dots,n\};
    \end{eqnarray*} 
then $$\sgn(f^n(0)-f^n(1))\neq \sgn(\beta^n(0)-\beta^n(1))=-1$$
hence from (A),(B) follows $\beta\notin B(\alpha,\eps)$, a contradiction.

\noindent   {\bf{II.}} $\sgn(f^n(0)-f^n(1))\in\{0,-1\}$. In this case there exists a window perturbation $\beta$ of $\alpha$ on the intervals $[0,\delta]$, $[1-\delta,1]$ such that $\alpha$ and $\beta$ coincide on $S_1$ and $T_0$, hence 
\begin{itemize}
    \item $\beta^i(0)=\alpha^i(x_1)$ for each $i\in\{1,\dots,n\}$ and also 
    \item $\beta^i(1)=\alpha^i(y_0)$ for each $i\in\{1,\dots,n\}$;
    \end{itemize} 
then $$\sgn(f^n(0)-f^n(1))\neq \sgn(\beta^n(0)-\beta^n(1))=1$$
and from (A),(B) one again gets $\beta\notin B(\alpha,\eps)$, a contradiction.

We have shown that the set $\overline{G}_f$ has an empty interior, i.e., $G_f$ is nowhere dense. 
\end{proof}
 
A subset $A$ of a Baire space $X$ has the {\it Baire property} if there is an open set $U \subset X$ such that the symmetric difference
$A \triangle U$ is meager.
 A subset of $X$ is called {\em analytic} if it is a continuous image of a Polish space. 
It is well known that every analytic set in a Polish space has the Baire property, e.g., see \cite[Theorem 4.3.2]{sriva}.

Let us define the metric $\tilde{\rho}$ on $\mathcal{H}(I)$ as follows. For $h_1,h_2\in \mathcal{H}(I)$ 
\begin{equation*}
\tilde{\rho}(h_1,h_2)=\rho(h_1,h_2)+\rho(h_1^{-1},h_2^{-1}).
\end{equation*}
Note that the metric $\tilde{\rho}$ is a complete metric in $\mathcal{H}(I)$.

\begin{lemma}\label{lem:Rbaire}
The set $\{(x,y): x\mathcal{R}y\}$
is an analytic subset of $\CP\times \CP$. In particular it has the Baire property.
\end{lemma}

\begin{proof}
Let 
$$A:=\{(f,g,h)\in (\CP)^2\times \mathcal{H}(I) : \text{ where }  g=h^{-1}\circ f\circ h \}$$ where we use the metric $\rho$ on $\CP$ and $\tilde{\rho}$ on $\mathcal{H}(I)$.  
Note that $A$ is a closed set in $(\CP)^2\times \mathcal{H}(I)$. 
The 
projection of $A$ onto $(\CP)^2$ yields the set $\{(x,y): x\mathcal{R}y\}$.
By (14.3) from Kechris~\cite{Kechris} 
the set 
$\{(x,y): x\mathcal{R}y\}$
is analytic is equivalent to the fact that there is a Polish space $Y$ and $B\subset (\CP^2)\times Y$ Borel such that $\mathcal{R}=\pi_{\CP^2}(B)$. Using the if direction
with $Y = \mathcal{H}(I)$ yields that the set
$\{(x,y): x\mathcal{R}y\}$ is analytic.
\end{proof}

Let $E$ be an equivalence relation on a Polish space $X$. We say that there
are {\it perfectly many $E$-equivalence classes} if there is a perfect set $A\subset X$ such that for any $x,y\in A$, $xEy$ implies $x = y$.
We obtain the following 
consequence of Theorem \ref{conj-nowhere}.

\begin{corollary}\label{cor:conjugacyclasses}
Every second category set $G\subset \CP$ contains uncountably many non-conjugate maps. The set $\CP \times \CP$
contains perfectly many conjugacy classes. 
\end{corollary}

\begin{proof}  The first assertion follows from Theorem \ref{conj-nowhere} since
each union of countably many conjugacy classes is a set of the first category. 
Since $\{(f,g) \in (\CP)^2: f\R g\}$  has the Baire property, then \cite[Corollary 5.3.3]{SG} shows that
this set contains perfectly many $\R$ conjugacy classes.
\end{proof}

\begin{remark}\label{rem:LebegueConjugacyClasses}
Adapting the preceding two proofs easily leads to results analogous to Theorem~\ref{conj-nowhere} and Corollary~\ref{cor:conjugacyclasses} in the setting of $C_{\lambda}$ or $CP$. By Remark~\ref{rem:psihomeo} the conclusions of Theorem~\ref{conj-nowhere} and Corollary~\ref{cor:conjugacyclasses} hold also for $C_{\mu}$. 
\end{remark}

 In \cite{CO} the authors constructed a parameterized family of planar homeomorphisms arising from a residual set $\mathcal{T}$ of functions from $C_{\lambda}$. The maps from $\mathcal{T}$ share several dynamical properties (shadowing, topological exactness, $\delta$-crookedness, the same structure of the sets of periodic points, etc.), hence are  indistinguishable by numerous conjugacy invariants. In this context the question
  if $\mathcal{T}$ could be contained in a single conjugacy class naturally arose.  Corollary~\ref{cor:conjugacyclasses} and Remark~\ref{rem:LebegueConjugacyClasses} provide a negative answer to this question.

\subsection{Shadowing is generic in $\CP$}\label{subsec:shadowing}

Now let us recall the definition of shadowing and periodic shadowing. For some $\delta> 0$, a sequence $(x_n)_{n\in \N_0}\subset I$ is called \emph{$\delta$-pseudo orbit} of $f\in C(I)$ if $d(f(x_n), x_{n+1})< \delta$ for every $n\in \N_0$. A \emph{periodic $\delta$-pseudo orbit} is a $\delta$-pseudo orbit for which there exists $N\in\N_0$ such that $x_{n+N}=x_n$, for all $n\in \N_0$. We call the sequence $(x_n)_{n\in\N_0}$ an \emph{asymptotic pseudo orbit} if $\lim_{n\to\infty} d(f(x_n),x_{n+1})=0$.

\begin{definition}\label{def:shadowing}
We say that a map $f\in C(I)$ has the:
\begin{itemize}
 \item \emph{shadowing property} if for every $\eps > 0$ there exists $\delta >0$ satisfying the following condition: given a $\delta$-pseudo orbit $\mathbf{y}:=(y_n)_{n\in \N_0}$ we can find a corresponding point $x\in I$ which $\eps$-traces $\mathbf{y}$, i.e.,
$$d(f^n(x), y_n)<  \eps \text{ for every } n\in \N_0.$$
\item
\emph{periodic shadowing property} if for every $\eps>0$ there exists $\delta>0$ satisfying the following condition: given a periodic $\delta$-pseudo orbit $\mathbf{y}:=(y_n)_{n\in\N_0}$ we can find a corresponding periodic point $x \in I$, which $\eps$-traces $\mathbf{y}$.
\end{itemize}
\end{definition}

\begin{lemma}[Lemma 22 from \cite{BCOT}]\label{lem:shadowingdense}
For every $\eps>0$ and every map $f\in C_\lambda$
 there are
$\delta<\frac{\eps}{2}$ and $F\in C_{\lambda}$ such that:
\begin{enumerate}
\item $F$ is piecewise affine and $\rho(f,F)<\frac{\eps}{2}$,
\item if $g\in C_\lambda$ and $\rho(F,g)<\delta$ then every $\delta$-pseudo orbit $\mathbf{x}:=\{x_i\}_{i=0}^\infty$ for $g$ is $\eps$-traced by a point $z\in I$.
Furthermore, if $\mathbf{x}$ is a periodic sequence, then $z$ can be chosen to be a periodic point.
\end{enumerate}
\end{lemma}

In fact, a close study of the proof  of Lemma 22 from \cite{BCOT} reveals that in statement (2) we may consider wider neighborhood of maps $g\in C(I)$, $\rho(F,g)<\delta$. We explain this in more detail in the following paragraph.

For an interval $J\subset I$ let $\diam(J):=\sup\{d(x,y): x,y\in J\}$.
Take any $g\in C_\lambda$ such that $\rho(F,g)<\delta$ and let $\mathbf{x}:=\{x_i\}_{i=0}^\infty$ be a $\delta$-pseudo orbit for $g$. In the proof we claim that there is a sequence of closed intervals $J_i$ such that
\begin{enumerate}[(a)]
\item $\diam J_i \leq \gamma$ and if $i>0$ then $J_{i}\subset g(J_{i-1})$,\label{con:s1}
\item $\mathrm{dist}(x_{i},J_i)< \gamma$,\label{con:s2}
\item\label{c:3} for every $i$ there is $p$ such that $F(J_i)=F([a_p,a_{p+1}])$ and $x_i\in [a_p,a_{p+1}]$\label{con:s3}
\end{enumerate}
provided that $g\in C_\lambda$ and $\rho(F,g)<\delta$,
where points $a_i$ are endpoints of a partition of $I$ and $\gamma>0$ is a sufficiently small constant.
The claim is proved simply by the definition of uniform metric and covering of intervals (by continuity of $F$), so in fact the conditions hold for any 
$g\in C(I)$, $\rho(F,g)<\delta$, in particular for $g\in \overline{CP}$. Now the tracing point is obtained as any $z\in \cap g^{-n}(J_n)\neq \emptyset$.

Therefore perturbing the original map $f\in \CP$ to a map $\hat{f}\in CP$, passing through conjugacy of $\hat f$ from $CP$ into $C_\lambda$, applying Lemma~\ref{lem:shadowingdense} (with sufficiently decreased constants) and then passing back  to $CP$ (which gives coverings satisfying \ref{con:s1}) we obtain the following adaptation of the result to the $CP$ case. 

\begin{lemma}\label{lem:shadowingdense2}
For every $\eps>0$ and every map $f\in \CP$
 there are
$\delta<\frac{\eps}{2}$ and $F\in CP$ such that:
\begin{enumerate}
\item $\rho(f,F)<\frac{\eps}{2}$,
\item if $g\in \CP$ and $\rho(F,g)<\delta$ then every $\delta$-pseudo orbit $\mathbf{x}:=\{x_i\}_{i=0}^\infty$ for $g$ is $\eps$-traced by a point $z\in I$.
Furthermore, if $\mathbf{x}$ is a periodic sequence, then $z$ can be chosen to be a periodic point.
\end{enumerate}
\end{lemma}

In what follows, we will need a result by Chen \cite{Ch91}, which connects the following definition with the shadowing property for interval maps.
\begin{definition} Let $f \in C(I)$ and fix any $\eps>0$. A point $x\in I$ is {\em $\eps$-linked to a point $y\in I$ by $f$} if there exists an integer $m > 1$ and a point $z \in B(x, \eps)$ such that $f^m(z) = y$ and
$d(f(x), f(z)) <\eps$ for $0< j< m$.
We say {\em $x \in I$ is linked to $y \in I$ by $f$} if x is $\eps$-linked to $y$ by $f$ for every $\eps > 0$. We say  {\em $C\subset I$ is linked by $f$} if every point $c \in  C$ is linked to some point in $C$.
We say that $f$ has the linking property if the set $C$ consisting of critical points together with endpoints of $I$ is linked by $f$.
\end{definition}

The following result is a consequence of \cite{Ch91} and \cite{Parry}.
\begin{theorem}\label{thm:chen}
For a topologically mixing piecewise monotone map of the interval, the shadowing property and the linking  property are equivalent.
\end{theorem}

\begin{theorem}\label{thm:shadowing}
Shadowing is generic in $CP$ and in  $\CP$ and there is a dense set of maps in $CP$  which do not have the shadowing property.
\end{theorem}

\begin{proof}
The proof that shadowing property is generic follows exactly the same arguments as the proof of Theorem~3 in \cite{BCOT} with Lemma~\ref{lem:shadowingdense} replaced by Lemma~\ref{lem:shadowingdense2}. 

To see that shadowing property is not dense property in $\CP$, fix any map $f\in CP$ and conjugate it by a homeomorphism $\psi$ with a map $F\in C_\lambda$. Since by Remark~\ref{rem:leolambda} $C_\lambda$ contain an open and dense subset $U$  of leo maps and piecewise affine maps form a dense subset of $C_\lambda$, we  can choose a piecewise affine leo  map $G$ in $U$ which is in  an arbitrarily small neighborhood of $F$. Now, we can perturb  map $G$ by a $2$-fold perturbation in a small neighborhood of $0$
to a piecewise affine map $\hat G$ such that $0$ is eventually periodic with periodic orbit $P$ and $P$ is disjoint from the set of turning points of  $\hat G$. That is indeed possible, because the set of preimages of any point $p$ under $G$ is dense in $I$, since $G$ is leo. If the perturbation is sufficiently small then $\hat G$ remains in $U$ and so it is leo. But it is clear that $\hat G$ does not have the linking property, so also does not have the shadowing property by Theorem~\ref{thm:chen}. If $\hat G$  is sufficiently close to $F$ then its conjugate $g$ by $\psi^{-1}$  is sufficiently close to $f$.  Observe that by definition both $\hat G$ and $g$ belong to $CP$.
This proves the density of maps without  the shadowing  property, completing the proof.
\end{proof}

\begin{remark}
Using Lemma~\ref{lem:shadowingdense} instead of Lemma~\ref{lem:shadowingdense2} we obtain with analogous arguments as in the proof of Theorem~\ref{thm:shadowing} that there is a dense set of maps in $C_\lambda$ which do not have the shadowing property.
\end{remark}

\subsection{Periodic point structure in $\CP$}\label{subsec:periodicpoints}
For $f\in C(I)$ denote 
$$\fix(f,k) := \{x: f^k(x) = x\}$$
$$\per(f,k) := \{x : f^k(x) = x \text{ and } f^i(x) \ne x  \text{ for all } 1 \le i < k
\}$$
$$k(x) := k \text{ for } x \in \per(f,k)$$
and
$$\per(f) := \bigcup_{k \ge 1} \per(f,k) = \bigcup_{k \ge 1} \fix(f,k).$$

\begin{definition}\label{transverse} 
We call a periodic point $p \in \per(f,k) $ \emph{transverse} if  there exist three adjacent intervals
$A = (a_1,a_2), B = [a_2,c_1], C = (c_1,c_2)$, with $p \in B$, $B$ possibly reduced to a point,
so that (1)  $f^k(x) = x$ for all $x \in B$ and either (2.a) $f^k(x) > x$ for all $x \in A$ and $f^k(x) < x$ for all $x \in C$ or (2.b)  $f^k(x) <x$ for all $x \in A$ and $f^k(x) > x$ for all $x \in C$.
\end{definition}

\begin{theorem}\label{thm:periodicpoints}
	For a generic map $f \in \CP$ and  for each $k \geq 1$:
	\begin{enumerate}
		\item \label{pp1} the set $\fix(f,k)$ is a Cantor set, 
		\item  \label{pp2} $\per(f,k)$ is a relatively open dense subset of $\fix(f,k)$, 
		\item \label{pp3}  the set $\fix(f,k)$ has Hausdorff dimension and  lower box dimension zero. In particular, $\per(f,k)$ has Hausdorff dimension and lower box dimension zero. As a consequence, the Hausdorff dimension of $\per(f)$ is also zero.
		 \item \label{pp4} the set $\per(f,k)$ has upper box dimension one. Therefore, \\
  $\fix(f,k)$ has upper box dimension one as well.
	\end{enumerate}
\end{theorem}

\begin{proof}
First note that  the last part of (\ref{pp3}) follows from the first part of (\ref{pp3}) since 
 $$\mathrm{dim}_H(\per(f)) \le \sup_{k \ge 1} \mathrm{dim}_H(\fix(f,k)) \le \sup_{k \ge 1} \underline{\mathrm{dim}}_{Box}(\fix(f,k)) = 0.$$

For the proofs of \ref{pp1}), \ref{pp2}), \ref{pp3}) and \ref{pp4}) it suffices to prove the result for a  fixed $k \in \N$.
 Fix a countable dense set of maps $\{f_i\}$ in $CP$. 
Let $\psi_i$ be a homeomorphism of $I$ such that 
$\hat{f}_i = \psi_i^{-1} \circ f_i \circ \psi_i$ belongs to $C_{\lambda}$. Let $\mathrm{PA}_{\lambda}$ denote the set of piecewise affine Lebesgue measure-preserving interval maps.
The set of maps in $\mathrm{PA}_{\lambda}$ having  
\begin{itemize}
\item[{a)}] no interval with slope $\pm 1$ and  
\item[{b)}]  having at least one periodic point of period $k$, and 
\item[{c)}]  all periodic points of period $k$ transverse
\end{itemize}
is dense in $C_{\lambda}$  (Definition 10 and Lemma 12 in \cite{BCOT}).
The set $\fix(\hat{g},k)$ is finite for such a map $\hat g$, since $\hat g$ does not have slope $\pm 1$.  
Thus we can choose $\hat{g_i} \in \mathrm{PA}_{\lambda}$  satisfying  a), b) and c) close enough to $\hat f_i$
and piecewise  affine homeomorphisms $\hat \psi_i$ of $I$ sufficiently close to  $\psi_i$ such that
\begin{itemize}
\item[(i)] each periodic point of $g_i$ is a point where  $\psi_i$ is differentiable  and 
\item[(ii)] the set of piecewise affine maps $\{g_i :=\hat  \psi_i \circ \hat g_i \circ \hat \psi_i^{-1}\}$ is 
a dense subset of $CP$. 
\end{itemize}
Notice that all points in  $\fix(g_i,k)$ are also transverse.
The advantage of transversality is that for each point in $\fix(g_i,k)$, there is at least one  corresponding periodic
point in $\fix(g,k)$ if the perturbed map $g$ is sufficiently close to $g_i$.

  After these modifications the rest of the proof  follows exactly the same arguments as the proof of Theorem~1 in \cite{BCOT}. For convenience we repeat and streamline the proof, correcting an error in the proof of \eqref{pp2} of Theorem~1 in \cite{BCOT} along the way.

Since $\gh_i \in  \mathrm{PA}_{\lambda}$ does not have slope $\pm 1$ the set $\fix(\hat{g}_i,k)$ is finite, and thus the set $\fix(g_i,k)$ is finite as well. Suppose it consists of $\ell_i$ disjoint orbits and  the set $\per(g_i,k)$ consists of $\bar \ell_i \le \ell_i$ distinct orbits.
In particular 
\begin{equation}\label{e-count}
\ell_i \le \# \fix(g_i,k) \le k \ell_i \text{  and  } \per(g_i,k) = k \bar \ell_i.
\end{equation}
Choosing one point  from each of the orbits  in $\fix(g_i,k)$ defines the set
 $\{x_{l,i}: 1 \le l \le \ell_i \} \subset \fix(g_i,k)$.  Let $k(x_{l,i})$ denote the minimal period of $x_{l,i}$.
Let $\hat{x}_{l,i} := \psi_i^{-1}(x_{l,i})$.

The construction in the proof depends on integers $n_i \ge 1$ which will be defined in the proof, for most of the estimates it suffices to have
$n_i =1$, but for the upper box dimension estimates we will need $n_i$ growing sufficiently quickly.
We consider a very small number $\hat a_i$ (several restrictions will be introduced progressively) and define new maps
$h_i := \psi_i \circ \hat{h}_i \circ \psi_i^{-1}$ where $\hat{\psi}_ i\in  \mathrm{PA}_{\lambda}$ is obtained by 
applying a regular $2n_i + 1$-fold window 
perturbation of $\hat{g}_i$ 
with respect to $\hat{I}_{l,i} := 
(\hat{x}_{l,i} - \hat a_i,\hat{x}_{l,i} + \hat a_i)$.
Finally set $I_{l,i} := \psi_i(\hat{I}_{l,i})$, and let $a_i$ be half the length of this interval.

We choose  $\hat a_i$  so small that $\{h_i\}$ satisfies the following properties:
\begin{enumerate}[i)]
\item $\psi_i|_{I_{l,i}}$ is affine.

\item  $2a_{i} \le (k \ell_{i})^{-i}$.\label{p5}
This implies $a_i \to 0$ and thus the collection $\{h_i\}_{i\geq 1}$ is dense in $CP$.

\item (Disjointness) 
$I_{l,i} \cap I_{l',i} = \emptyset$ if $l \ne l'$
and for all $1 \le l < \ell_i $ and all $0 \le j < k$ the
images $h_i^{j} (I_{l,i})$ are mutually disjoint,\label{p1} 

\item \label{p2}   (Regularity) For each $1 \le j \le k(x_{l,i})$ the map
$h^j_i$ restricted to the interval $I_{l,i}$ is composed of $2n_i+1$ full monotone branches
with widths equal to $2a_i/(2n_i + 1)$. This implies
\begin{enumerate}[a)]
\item \label{p2a}
 The map $h_i^{k(x_{l,i})}$  has exactly  $2n_i + 1$ fixed points
in the interval  $I_{l,i}$ and these points have period $k(x_{l,i})$,
\item \label{p2b} The map $h_i^k$ has
 $(2n_i+1)^{k/k(x_{l,i})}$ fixed points in this interval.
 \item \label{p2c}  The full branches of $h_i^{k/k(x_{l,i})}$ have width $b_{l,i} := 2a_i/((2n_i+1)^{k/k(x_{l,i})})$, thus
each subinterval of $I_{l,i}$ of length $2b_{l,i}$ contain at least one full branch and at most parts of three branches,  and thus at least one fixed point and at most 3 fixed points of $h_i^{k/k(x_{l,i})}$.
 \item\label{p2d}  The map $h_i$ has a point of period $k$ in each interval $I_{l,i}$.

 \item \label{p2e} The total number $N_{l,i}$ of fixed points of $h_i^k$ arising from the orbit of $x_{l,i}$  satisfies
$$ N_{l,i}   =  (2n_i + 1)^{k/k(x_{l,i})} k(x_{l,i}).$$
Summing over the points $x_{l,i}$  and using $1 \le k(x_{l,i}) \le k$ yields
$$\#  \fix(h_i,k) = \sum_{l=1}^{\ell_i} N_{l,i}   \le (2n_i + 1)^{k} k \ell_i.$$
Since we have $\ell_i$ different points $x_{l,i}$ and among them there is a fixed point of $g_i$, so there is at least one $l$ with $k(x_{l,i})=1$, yielding
$$\max( (2n_i + 1) \ell_i ,  (2n_i + 1)^k)
 \le \#  \fix(h_i,k).$$
 \end{enumerate}
\item new periodic points obtained by perturbations of each $x_{l,i} \in \per(g_i,k)$ have to
visit all $k$ disjoint intervals $h_i^j (I_{l_0,i})$  $(0 \le j < k)$.
Thus for each $x_{l,i} \in \per(g_i,k)$ the $N_{l,i} = (2n_i + 1)k$ points are
not only in $\fix(h_i,k)$ but also in  $\per(h_i,k)$; and so
 $\# \per(h_i,k) \ge (2n_i + 1)k \bar \ell_i$. 
\end{enumerate}

Consider $\delta_i > 0$ and
 $$\mathcal{G}:= \bigcap_{j \in\N} \bigcup_{i \ge j} B(h_i ,\delta_i),$$
 where the ball is taken in $\CP$.
 By construction the set $\mathcal{G}$ is a dense $G_{\delta}$ subset of $\CP$,  by Lemma~\ref{l1} the set $\mathcal{G} \cap CP$ is a $G_\delta$ subset of $CP$.
 
 {\it (\ref{pp1})} First we claim that if $\delta_i > 0$ goes to zero sufficiently quickly then $\fix(f,k)$ is a Cantor set for each $f \in \mathcal{G}$.  
 
By its definition the set $\fix(f,k)$ is 
closed for any continuous map $f$.
The set $\fix(h_i,k)$ is finite.
We choose $\delta_i$ so small that if $g \in B(h_i,\delta_i)$ then $\fix(g,k) \subset  B(\fix(h_i,k),a_i)$;  in particular the set
$\fix(g,k)$ can not contain an interval whose length is longer than $2a_i$.
Suppose $f \in \mathcal{G}$,  thus $f \in B(h_{i_j},\delta_{i_j})$  for some subsequence $i_j$.
Since $a_{i _j}\to 0$ the set $\fix(f,k)$ can not contain an interval.

It remains to show that 
there are no isolated points in $\fix(f,k)$. 
By (ii\ref{p2b} each interval $I_{l,i}$ contains exactly
$ (2n_i+1)^{k/k(x_{l,i})}$ points of $\fix(h_i,k)$.
 We choose $\delta_i$ so small
that  
 for each $g \in B(h_i,\delta_i)$
each interval $I_{l,i}$ contains at least $ (2n_i+1)^{k/k(x_{l,i})}$ points of $\fix(g,k)$  and
 $\fix(g,k) \subset \cup_{l,0 \le j \le k(x_{l,i})} h_i^j(I_{l,i})$.
Now we apply  this to $f \in \mathcal{G}$ and $x \in \fix(f,k)$. Let $l_j$ be such that $x \in I_{{l_j},i_j}$. Each of these intervals contains other points of $\fix(f,k)$, and their
lengths tend to zero, which show that $x$ is not isolated and 
 completes the proof of \eqref{pp1}.

{\it (\ref{pp2})} Fix $f\in \mathcal{G}$  
and an $i$ such that 
$f \in B(h_i,\delta_i)$. 
As we already saw in the proof of  \eqref{pp1}  we have  $\fix(f,k) \subset  B(\fix(h_i,k),a_i)$. 

But (ii\eqref{p2d} implies that
$\fix(h_i,k)\subset B(\per(h_i,k),2a_i)$.
Combining these two facts yields 
$\fix(f,k) \subset  B(\fix(h_i,k),a_i) \subset B(\per(h_i,k),3a_i)$.

Next we claim that $\delta_i$ can be chosen so that $\per(h_i,k) \subset B(\per(f,k),a_i)$. Once the claim is proven we will have shown that $\per(f,k)$ is dense in $\fix(f,k)$.

Condition  (ii\ref{p2a} tells us that  if $x_{l,i} \in \per(h_i,k)$  then the $(2n_i+1)$ fixed points of $h_i^k$ in the interval $I_{l,i}$   all have period $k$.
We have chosen $\delta_i$ so small
that each $I_{l,i}$ contains at least as many points of $\fix(g,k)$  for any $g \in B(h_i,\delta_i)$ and
 $\fix(g,k) \subset \cup_{l,0 \le j \le k(x_{l,i})} h_i^j(I_{l,i})$.
 We additionally require that $\delta_i$ is so small that  a similar statements hold for periodic points, namely that 
 if $k(x_{l,i}=k$ then each $I_{l,i}$ contains at least as many points of $\per(g,k)$  for any $g \in B(h_i,\delta_i)$ and 
  $\per(g,k) \subset \cup_{l : k(l,i) = k ,0 \le j \le k} h_i^j(I_{l,i})$.
This implies $\per(h_i,k) \subset  B(\per(g,a_i)$ as well, and finishes the proof of the claim.

Finally notice that $\per(f,k)= \fix(f,k)\setminus \cup_{1 \le \ell < k: \ell | k} \fix(f,\ell)$.
But this finite union is closed, thus $\per(f,k)$ is relatively open subset of $\fix(f,k)$. This completes the proof of \eqref{pp2}.

{\it \ref{pp3})}  It suffices to prove the lower box dimension statement for $\fix(f,k)$, the other statements follow since  $\per(f,k) \subset \fix(f,k)$ and  since the Hausdorff dimension
of a set is smaller than its lower box dimension.
We will show  that if  $\delta_i > 0$ goes to zero sufficiently quickly   then
the lower  box dimension of $\fix(f,k)$ is zero for any $f \in  \mathcal{G}$.

Consider the open cover $\mathcal{C}_i := \{(x-a_i,x+a_i): x \in \fix(h_i,k)\}$ of $\fix(h_i,k)$, notice that
the intervals are pairwise disjoint by assumption (\ref{p1} and that $\mathcal{C}_i$  covers the domain $\{I_{l,i}: 1 \le l \le l_i\}$ of our perturbations.
Choose $\delta_{i}$ sufficiently small so that  for each $f \in B(h_i,\delta_{i})$ the collection $\mathcal{C}_i$ covers the  set
 $\fix(f,k)$.

Fix $f \in \mathcal{G}$, thus $f \in B(h_{i_j},\delta_{i_j})$  for some subsequence $i_j$.
Let $N(\eps)$ denote the number of intervals of length $\eps > 0$ needed to cover $\fix(f,k)$, combining the previous discussion with   
Equation \eqref{e-count} yields
$N(2a_{i_j}) \le k \ell_{i_j} = \# \mathcal{C}_i$.  Applying (\ref{p5} yields
$$\frac{\log (N(2a_{i_j})) }{\log(1/2a_{i_j})}  \le  \frac{ \log( k\ell_{i_j})}{\log(1/2a_{i_j})} \le \frac{1}{i_j}$$
and thus the lower box dimension of $\fix(f,k)$
defined by
$$\liminf_{\eps \to 0} \frac{\log (N(\eps)) }{\log(1/\eps)}$$
is 0.

{\it (\ref{pp4})}  Instead of covering $\fix(h_i,k)$ by intervals of length $a_i$  we cover it by intervals of length $2b_i$.  
By (ii\ref{p2c} each such interval covers at most three points of $\fix(h_i,k)$.
Thus we need at least $(\# \fix(h_i,k))/3$ such intervals to cover $\fix(h_i,k)$; so by (ii\ref{p2e} we need
 at least $(2n_i+1)^{k}/3$ such intervals to cover $\fix(h_i,k)$.
Fix such a covering and  choose $\delta_i$ sufficiently small so that all periodic points  of period $k$ of any $f \in B(h_{i},\delta_{i})$ are contained in  the covering intervals.

Thus 
\begin{equation}\label{ubd}
\frac{\log (N(2b_{i})) }{\log(1/2b_{i})} \ge  \frac{ \log((2n_i+1)^k/3)}{\log(1/2b_{i})} 
=   \frac{ \log((2n_i+1)^k) - \log(3)}{\log((2n_i+1)^k) - \log(2a_i)}.
\end{equation}

The sequence $a_i$ has been fixed above, 
if $n_i$ grows sufficiently quickly the last term in \eqref{ubd} approaches one. 
We can not cover  $\fix(h_i,k) $ by fewer intervals, and thus we can not cover $\fix(f,k)$ by fewer intervals for any $f \in B(h_{i},\delta_{ i})$. 

By the above discussion, if we fix any $f\in \mathcal{G}$ then $f\in B(h_{i_j},\delta_{i_j})$ for some subsequence $i_j$ and therefore,
 the upper box dimension of $\fix(f,k)$
defined as 
$$\limsup_{\eps \to 0} \frac{\log (N(\eps)) }{\log(1/\eps)}$$
is 1.
\end{proof}

\begin{remark}
The statements of Theorem~\ref{thm:periodicpoints} also hold for generic maps in $CP$ by Lemma~\ref{l1}.
\end{remark}

\begin{theorem}\label{t-PP}
	The set of leo maps in $CP$
	whose periodic points have full Lebesgue measure and whose periodic points of period $k$ have positive measure for each $k \ge 1$
	is dense in $CP$.
\end{theorem}

\begin{proof}Fix a map $\alpha\in CP$. By \cite[Remark, p. 2536]{BCOT} the map  $\alpha$ is from $C_{\mu}$ for some nonatomic probability measure $\mu$ with $\supp~\mu=I$. Moreover, the map $\psi\colon~I\to I$ defined by $\psi(x)=\mu([0,x])$ is a homeomorphism for which $\beta=\psi\circ\alpha\circ \psi^{-1}\in C_{\lambda}$. By \cite[Theorem 2]{BCOT} the statement of Theorem~\ref{t-PP} is true in $C_{\lambda}$.  So there exists a map $\gamma\in C_{\lambda}$ for which  \begin{align}\label{a:11}&\rho(\beta,\gamma)<\varepsilon_1,~\lambda(\per(\gamma,k))>0~\text{ for each }k,\\&\lambda(\per(\gamma))=\lambda(\bigcup_{k\ge 1}\per(\gamma,k))=1.\nonumber
\end{align}
Obviously there is a piecewise affine homeomorphism $\ell\colon~I\to I$ such that \begin{equation}\label{e:11}\rho(\psi,\ell)+\rho(\psi^{-1},\ell^{-1})<\varepsilon_1.
\end{equation}

Choose $\varepsilon_2>0$ arbitrary. If $\varepsilon_1$ is sufficiently small, from (\ref{a:11}) and (\ref{e:11}) one has for $\delta=\ell^{-1}\circ \gamma\circ \ell$, $\per(\delta,k)=\ell^{-1}(\per(\gamma,k))$ for each $k$ hence also $\per(\delta)=\ell^{-1}(\per(\gamma))$,
\begin{align*}\rho(\alpha,\delta)=\rho(\psi^{-1}\circ\beta\circ \psi,\ell^{-1}\circ\gamma\circ\ell)<\varepsilon_2,\end{align*}
and since $\ell$ is piecewise affine, 
$$
\lambda(\per(\delta,k))>0\text{ for each}~k  \text{ and } \lambda(\per(\delta))=1. 
$$
\par \vspace{-2.2\baselineskip}  \qedhere
\end{proof}

For completeness we also show the following proposition which is well known in many situations. Let 
 $CP_{entr=\infty}$ denote the set of maps from $CP$ with topological entropy $\infty$.

\begin{proposition}\label{prop:entropyinfty}
The set of maps $CP_{entr=\infty}$ is a residual set in $CP$ and thus in $\CP$.
\end{proposition}

\begin{proof}
 Every map $f\in CP\setminus \{\mathrm{id}\}$ has a fixed point $b$ where the graph of $f$ is transverse to the diagonal at $b$. This is indeed true since it holds for maps in $C_{\lambda}\setminus \{\mathrm{id}\}$ and any map $f\in CP$ is conjugate to a map $\hat f=\psi^{-1}\circ f\circ \psi\in C_{\lambda}$; therefore also $f\in CP\setminus \{\mathrm{id}\}$ have a transverse fixed point. 
 Using an $(n + 2)$-fold window perturbation on a neighborhood of $b$, and again passing through $C_{\lambda}$, we can create a map $g\in CP$ arbitrarily close to $f$ with a horseshoe with entropy $\log(n)$ in the window. Since horseshoes are stable under perturbations around stable fixed points, there is an open ball $B(g, \delta)$ in $CP$ such that each $h$ in this ball has topological entropy at least $\log(n)$ for any $n\geq 1$. The result also hold for $\CP$ by Lemma~\ref{l1}.
\end{proof}

\section{Acknowledgments}
We would like to thank Udayan B. Darji for his comments and suggestions which helped improving the exposition of the article.
 J. \v Cin\v c was partially supported by the Slovenian research agency ARIS grant J1-4632 and by the European Union's Horizon Europe Marie Sk\l odowska-Curie grant HE-MSCA-PF- PFSAIL - 101063512. J. Bobok was supported by the European Regional Development Fund, project No.~CZ 02.1.01/0.0/0.0/16\_019/0000778.
 P. Oprocha was supported by National Science Centre, Poland (NCN), grant no. 2019/35/B/ST1/02239.
 
 \vspace{0.7cm}

\begin{table}[ht]
\begin{tabular}[t]{p{1.5cm}  p{12cm} }
\includegraphics [width=.09\textwidth]{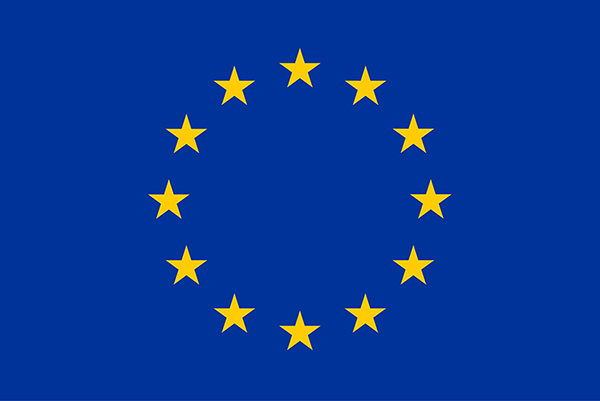} &
\vspace{-1.3cm}
This research is part of a project that has received funding from the European Union's Horizon Europe research and innovation programme under the Marie Sk\l odowska-Curie grant agreement No 101063512.\\
\end{tabular}
\end{table}

\end{document}